\documentclass[11pt]{article}
\usepackage{graphicx,amsmath,amsfonts,amssymb,bbm,amsthm}
\usepackage{psfrag}
\usepackage{epsf}

\makeatletter\@addtoreset {equation}{section}\makeatother

\theoremstyle{plain}
\newtheorem{theorem}{Theorem}[section]
\newtheorem{lemma}[theorem]{Lemma}

\newtheorem{proposition}[theorem]{Proposition}
\newtheorem{corollary}[theorem]{Corollary}
\newtheorem{assumption}[theorem]{Assumption}
\theoremstyle{remark}
\newtheorem{remark}[theorem]{Remark}

{\hspace*{\fill}$\rule{.3\baselineskip}{.35\baselineskip}$\end{trivlist}}

\newcommand{\R}{\mathbb{R}}

\renewcommand{\phi}{\varphi}

\graphicspath{{./Figures/}}
\usepackage{color}
\usepackage{bm}

\newcommand{\bmU}{\bm{U}}
\newcommand{\bmv}{\bm{v}}
\newcommand{\bmz}{\bm{z}}
\newcommand{\bmk}{\bm{k}}
\newcommand{\bmkk}{\bm{k_1}}
\newcommand{\AP}{A_\gamma}
\newcommand{\BP}{B_\gamma}
\newcommand{\PP}{P_{c,\gamma}}
\newcommand{\FF}{\mathcal{F}}
\newcommand{\II}{\bm{I}}

\begin{document}

\title{\bf Orbital stability of periodic waves \\
in the class of reduced Ostrovsky equations}

\author{Edward R. Johnson$^{1}$ and Dmitry E. Pelinovsky$^{2}$ \\
{\small $^{1}$ Department of Mathematics, University College London, London, United Kingdom, WC1E 6BT} \\
{\small $^{2}$ Department of Mathematics, McMaster University, Hamilton, Ontario, Canada, L8S 4K1}  }

\date{\today}
\maketitle

\begin{abstract}
Periodic travelling waves are considered in the class of
reduced Ostrovsky equations that describe low-frequency
internal waves in the presence of rotation.
The reduced Ostrovsky equations with either quadratic or cubic nonlinearities
can be transformed to integrable equations of the Klein--Gordon type
by means of a change of coordinates.
By using the conserved momentum and energy  as well as an additional
conserved quantity due to integrability, we prove that small-amplitude periodic
waves are orbitally stable with respect to subharmonic perturbations,
with period equal to an integer multiple of the period of the wave.
The proof is based on construction of a Lyapunov functional, which is convex
at the periodic wave and is conserved in the time evolution.
We also show numerically that convexity of the Lyapunov functional
holds for periodic waves of arbitrary amplitudes.
\end{abstract}

\section{Introduction}
\label{sec:intro}

The Ostrovsky equation with the quadratic nonlinearity was originally derived by L.A. Ostrovsky
\cite{Ostrov} to model small-amplitude long waves in a rotating fluid of finite depth.
The same model was extended to include the cubic nonlinearities in the context of
the internal gravity waves \cite{Grimshaw1,Grimshaw2}.

We consider the class of reduced Ostrovsky equations, for which 
the high-frequency dispersion is neglected. In particular,
we are interested in the models with either quadratic or cubic
nonlinearities. When a suitable scaling is selected, the
two different evolution equations can be written in the normalized forms
\begin{equation}
\label{redOst}
(u_t + u u_x)_x = u
\end{equation}
and
\begin{equation}
\label{redModOst}
\left( u_t + \frac{1}{2} u^2 u_x \right)_x = u,
\end{equation}
where $u(x,t) ; \R \times \R \to \R$. The modified reduced Ostrovsky equation (\ref{redModOst})
can be considered as the defocusing version of the short pulse equation
\begin{equation}
\label{shortPulse}
\left( u_t - \frac{1}{2} u^2 u_x \right)_x = u.
\end{equation}
The short-pulse equation was derived in \cite{Schafer1,Schafer2,PelSchneider} in the context of propagation
of ultra-short pulses with few cycles on the pulse width.

Local solutions of the Cauchy problem associated with the class of reduced Ostrovsky equations
exist in the space $H^s(\mathbb{R}) \cap \dot{H}^{-1}(\mathbb{R})$ for $s > \frac{3}{2}$ \cite{SSK10}. For the sufficiently
large initial data, the local solutions break in a finite time, similar to the inviscid Burgers equation
with quadratic or cubic nonlinearities \cite{LPS1,LPS2}.
However, if the initial data $u_0$ is small in a suitable norm, then
local solutions are continued for all times in the same space.
For the quadratic equation (\ref{redOst}), the global solutions
exist if $u_0 \in H^3(\mathbb{R})$ with $1 - 3 u_0''(x) > 0$ for every $x \in \mathbb{R}$ \cite{GH,GP}.
For the cubic equation (\ref{redModOst}), the global solutions exist if $u_0 \in H^2(\mathbb{R})$
with $|u_0'(x)| < 1$ for every $x \in \mathbb{R}$ \cite{JG,PelSak}.

This work is devoted to the orbital stability analysis of periodic travelling waves
of the reduced Ostrovsky equations (\ref{redOst}) and (\ref{redModOst}).
Periodic travelling waves can be found in the explicit form of the Jacobi elliptic functions,
after a change of coordinates \cite{GH}. Alternatively, the periodic waves can be characterized as
critical points of the conserved energy \cite{Pava}. If the periodic waves
are constrained minimizers of energy with respect to the perturbations of the same period
subject to the fixed conserved momentum,
then they are orbitally stable with respect to
such perturbations (see, e.g., the recent works \cite{Pava1,Pava2,Gallay,Haragus,Nataly1,Nataly2}).

Unfortunately, the periodic waves are not typically constrained minimizers of
energy with respect to perturbations with period equal to an integer multiple of the wave period
(such perturbations are called subharmonic). It becomes therefore difficult (if not impossible)
to prove the nonlinear orbital stability of periodic waves with respect to subharmonic perturbations
even if the spectral stability of the periodic waves is established (e.g., by computing
numerically the Floquet--Bloch spectrum of the linearization at the periodic wave).

However, the reduced Ostrovsky equations (\ref{redOst}) and (\ref{redModOst}) can be transformed
to the integrable equations of the Klein--Gordon type with a change of coordinates \cite{GH,JG}.
As a result, these evolution equations have a bi-infinite sequence of conserved quantities beyond
the conserved energy and momentum. By combining several conserved quantities in a linear combination,
one can try obtaining a conserved quantity, for which the periodic waves are unconstrained minimizers
with respect to subharmonic perturbations. These minimizers are degenerate due to the spatial translation symmetry,
but the degeneracy can be overcome by using the spatial translation parameter.

These ideas to obtain nonlinear orbital stability of periodic waves in integrable equations
have been developed in the works of B. Deconinck and his collaborators \cite{b1,Decon1,b2,Decon3,Decon2}.
Using the eigenfunctions of Lax operators arising in the inverse scattering method, 
a complete set of Floquet--Bloch eigenfunctions satisfying the linearization of the integrable equations
at the periodic wave is constructed and the quadratic forms associated with
the higher-order energy functionals are computed at the Floquet--Bloch eigenfunctions.
If the quadratic form for a linear combination of the higher-order energy functionals is shown to
be positive for every Floquet--Bloch eigenfunction of the linearized integrable equation,
then one can conjecture that the corresponding energy functional can be used as the Lyapunov function
in the proof of the orbital stability of unconstrained minimizers. In such a way,
the nonlinear orbital stability of periodic waves with respect to subharmonic perturbations
was obtained for the Korteweg--de Vries \cite{b1,Decon2,b2}, modified Korteweg--de Vries \cite{Decon3}, and the defocusing
nonlinear Schr\"{o}dinger \cite{Decon1} equations.

In the recent work \cite{GP1}, the proof of nonlinear orbital stability
was revisited for the periodic waves in the defocusing nonlinear Schr\"{o}dinger equation.
In particular, it was proven rigorously that the periodic waves are unconstrained minimizers
for a linear combination of two natural energy functionals, one of which is defined in $H^1$
and the other one is defined in $H^2$. Compared to the work \cite{Decon1},
the proof was developed directly for the Hessian operators associated with the higher-order Lyapunov
functional. Positivity of the corresponding quadratic form is proved for small-amplitude periodic waves by means of the
perturbation theory and for large-amplitude periodic waves by means of new factorization identities
for the defocusing NLS equation and a continuation argument.

In this work, we extend the nonlinear orbital stability analysis to the class of reduced Ostrovsky equations, which include
the integrable equations (\ref{redOst}) and (\ref{redModOst}). Similarly to the scopes of \cite{GP1},
we would like to work with a linear combination of two energy functionals, one of which is the energy of the reduced Ostrovsky
equations defined in $H^{-1} \cap L^p$ space with $p = 3$ for (\ref{redOst}) and $p = 4$ for (\ref{redModOst}).
The other energy functional is a subject of the free will. For instance, we could look at
the higher-order energy defined in a subset of $H^{-2}$ space or the higher-order energy
defined in $H^s$ with $s = 3$ for (\ref{redOst}) or $s = 2$ for (\ref{redModOst}).
However, we will show that the Casimir-type functional, which is the main building block
for integrability of the corresponding equations \cite{Br,BS}, is sufficient for the construction
of the Lyapunov functional for the periodic wave. The second variation of the Lyapunov functional
is positive for the $H^s \cap H^{-1}$ perturbations to the periodic wave with $s = 2$
for (\ref{redOst}) and $s = 1$ for (\ref{redModOst}), thus providing stability of the periodic waves
in the reduced Ostrovsky equations (\ref{redOst}) and (\ref{redModOst}).

We prove the main result for the small-amplitude periodic waves, when computations are performed by the perturbation
theory and do not require a heavy use of integrability technique. We also show numerically that convexity
of the Lyapunov functional extend to the periodic waves of larger amplitudes until the terminal
amplitude is reached, where the periodic wave profile becomes piecewise smooth.

The paper is organized as follows. The conserved quantities and the main stability result for the reduced
Ostrovsky equations are described in Section 2 for (\ref{redOst}) and in Section 3 for
(\ref{redModOst}). The general proof of positivity of a linear combination of two energy functionals
for the small-amplitude periodic waves
is developed in Section 4, where applications of the general method to the two integrable equations
(\ref{redOst}) and (\ref{redModOst}) are also given. Section 5 reports numerical results for the
periodic waves of arbitrary amplitudes. In the concluding Section 6, we discuss
how the stability analysis based on higher-order conserved quantities fails for the short-pulse equation (\ref{shortPulse}),
where the small-amplitude periodic waves are known to be unstable with respect to side-band modulations.

\section{Main result for the reduced Ostrovsky equation (\ref{redOst})}

In the context of the reduced Ostrovsky equation (\ref{redOst}), we
define travelling periodic wave solutions and conserved quantities,
compute the second variation of the energy functional at the periodic wave profile,
and present the main result on positivity of the second variation for a certain
linear combination of the energy functionals.

\subsection{Travelling waves}

Travelling $2L$-periodic solutions of the reduced Ostrovsky equation (\ref{redOst})
can be expressed in the normalized form
\begin{equation}
\label{trav-wave}
u(x,t) = \frac{L^2}{\pi^2} U(z), \quad z = \frac{\pi}{L} x - \frac{L}{\pi} \gamma t,
\end{equation}
where $U(z)$ is a $2\pi$-periodic solution of the second-order differential equation
\begin{equation}
\label{second-order}
\frac{d}{dz} \left[ (\gamma - U) \frac{dU}{dz} \right] + U(z) = 0,
\end{equation}
and the parameter $\gamma$ is proportional to the wave speed. Note that
the $2\pi$-periodic function $U$ satisfying (\ref{second-order}) has the zero mean value.
By the translational invariance and reversibility of the differential equation (\ref{second-order}),
$U$ can be selected to be an even function of $z$.

The second-order equation (\ref{second-order}) can be integrated with the first-order invariant
\begin{equation}
\label{first-order}
I = \frac{1}{2} (\gamma - U)^2 \left( \frac{dU}{dz} \right)^2 + \frac{\gamma}{2} U^2 - \frac{1}{3} U^3 = {\rm const}.
\end{equation}
The following result defines a family of the $2\pi$-periodic even solutions $U$
of the differential equation (\ref{second-order}) in the small-amplitude limit.

\begin{lemma}
\label{lemma-small-amplitude}
For every $\gamma > 1$ such that $|\gamma - 1|$ is sufficiently small,
there exists a unique $2\pi$-periodic, even, smooth solution $U$,
which can be parameterized by the amplitude parameter $a$, such that
\begin{equation}
\label{Stokes}
U(z) = a \cos(z) + \frac{1}{3} a^2 \cos(2z) + \frac{3}{16} a^3 \cos(3z) + \mathcal{O}_{C^{\infty}_{\rm per}}(a^4)
\quad \mbox{\rm as} \quad a \to 0,
\end{equation}
where parameter $a$ determines unique values of the parameters $\gamma$ and $I$
given by the asymptotic expansions
\begin{equation}
\label{gamma}
\gamma = 1 + \frac{1}{6} a^2 + \mathcal{O}(a^4) \quad \mbox{\rm as} \quad a \to 0
\end{equation}
and
\begin{equation}
\label{E-parameter}
I = \frac{1}{2} a^2 + \mathcal{O}(a^4) \quad \mbox{\rm as} \quad a \to 0.
\end{equation}
\end{lemma}

\begin{proof}
Justification of the asymptotic expansions (\ref{Stokes}), (\ref{gamma}), and (\ref{E-parameter}) 
and the proof of existence of the $2\pi$-periodic even solutions $U$ of the differential equation
(\ref{second-order}) can be achieved by the standard method
of Lyapunov--Schmidt reductions, see, e.g., the proof of Proposition 2.1 in \cite{GP1}. The solution $U$ is smooth both
in variables $z$ and $a$ because the differential equation (\ref{second-order}) is smooth
in $U$ and the wave profile satisfies $U(z) < \gamma$ for every $z \in \mathbb{R}$ and for small amplitudes $a$.

The representation (\ref{Stokes}) is typically referred to as the Stokes expansion
of a small-amplitude periodic wave. Since the period of the periodic wave has been normalized
to $2\pi$, the parameter $a$ of the wave amplitude
also parameterize the wave speed $\gamma > 1$ and the first-order invariant $I > 0$
according to the asymptotic expansions (\ref{gamma}) and (\ref{E-parameter}).

Computations of the coefficients in the expansions (\ref{Stokes}), (\ref{gamma}), and (\ref{E-parameter}) 
are straightforward and hence omitted. 
\end{proof}

\begin{remark}
In the statement of Lemma \ref{lemma-small-amplitude} and further on, we will use the following notations,
which are typical in the asymptotic analysis. If $I$ depends smoothly on the small parameter $a$,
then $I = \mathcal{O}(a^2)$ means that the limit $I/a^2$ as $a \to 0$ exists.
Similarly, if $U \in C^{\infty}_{\rm per}$ is a $2\pi$-periodic smooth function that also depends
smoothly on the small parameter $a$, then $U = \mathcal{O}_{C^{\infty}_{\rm per}}(a^2)$ means
that the limit $\sup_{z \in [0,2\pi]} |U^{(n)}(z)|/a^2$ as $a \to 0$ exists for every $n \in \mathbb{Z}$
(the limiting value as $a \to 0$ may depend on $n$).
\end{remark}

\begin{remark}
The $2\pi$-periodic solution of the second-order equation (\ref{second-order}) can be
expressed in the closed analytical form by using a change of coordinates \cite{GH}.
The transformation is given by
\begin{equation}
\label{wave-transformation}
U(z) = {\bf u}(\zeta), \quad z = \int_0^{\zeta} (\gamma - {\bf u}(\zeta')) d \zeta',
\end{equation}
where ${\bf u}$ satisfies the following second-order equation
\begin{equation}
\label{KdV-ODE}
\frac{d^2 {\bf u}}{d \zeta^2} + (\gamma - {\bf u}) {\bf u} = 0.
\end{equation}
The latter equation arises for travelling waves of the KdV equation and it admits
explicit solutions for periodic waves given by the Jacobi elliptic ${\rm cn}$ functions \cite{GH}.
The family of $2\pi$-periodic smooth solutions of the differential equation (\ref{second-order})
exists in variables $U$ and $z$ as long as the transformation (\ref{wave-transformation})
is invertible, that is, as long as ${\bf u}(\zeta) < \gamma$ for every $\zeta$.
\end{remark}

\begin{remark}
It follows from the recent work in \cite{Vill1,Vill2} that the amplitude-to-period map 
for the differential equations (\ref{second-order}) and (\ref{first-order}) is strictly monotonically increasing. 
With the scaling transformation, this result can be used to prove monotonicity
of the parameters $\gamma$ and $I$ in terms of the wave amplitude $a$, which are expanded
asymptotically as $a \to 0$ by (\ref{gamma}) and (\ref{E-parameter}).
\end{remark}

\begin{remark}
The periodic wave with profile $U$ exists for every $\gamma \in \left(1,\frac{\pi^2}{9}\right)$ \cite{GH}.
As $\gamma \to \frac{\pi^2}{9}$, the periodic wave is no longer smooth at the end points of the fundamental period.
In variables $U$ and $z$, the limiting wave has the parabolic profile:
\begin{equation}
\label{parabolic-wave}
\gamma = \frac{\pi^2}{9} : \quad U(z) = \frac{3 z^2 - \pi^2}{18},
\end{equation}
such that $U(\pm \pi) = \gamma$. As $\gamma \to \frac{\pi^2}{9}$, the parabolic wave profile (\ref{parabolic-wave})
yields the value
\begin{equation}
\label{first-order-final}
I = \frac{1}{6} \gamma^3 = \frac{\pi^6}{4374}.
\end{equation}
From the first-order invariant (\ref{first-order}) at the solutions of the second-order
equation (\ref{second-order}), we deduce that
\begin{equation}
\label{identity}
1 - 3 U'' = \frac{\gamma^3 - 6 I}{(\gamma - U)^3}.
\end{equation}
The $2\pi$-periodic solution $U$ of
the second-order equation (\ref{second-order}) with $\gamma \in \left(1,\frac{\pi^2}{9}\right)$
satisfies the constraint $\gamma^3 > 6 I$.
This follows from the identity (\ref{identity}) since $1 - 3 U''(z) > 0$ and
$U(z) < \gamma$ for all $z$ if $\gamma \in \left(1,\frac{\pi^2}{9}\right)$ \cite{GH}.
\end{remark}

\subsection{Conserved quantities}

Let us now recall \cite{LPS1} that local solutions $u$ of the reduced Ostrovsky equation (\ref{redOst})
in the space $H^s \cap \dot{H}^{-1}$ with $s > \frac{3}{2}$ has conserved momentum $Q(u) = \| u \|_{L^2}^2$
and energy
\begin{equation}
\label{energy}
E(u) = \| \partial_x^{-1} u \|_{L^2}^2 + \frac{1}{3} \int u^3 dx,
\end{equation}
where the integration is extended over the wave period and $\partial_x^{-1}$ denotes
the zero-mean anti-derivative of the function $u$. Writing all the integrals in
the normalized variable $z$ and neglecting the scaling factors (which is equivalent
to the choice $L = \pi$), we consider the first energy functional in the form
\begin{equation}
\label{first-Lyapunov}
S_{\gamma}(u) := E(u) - \gamma Q(u).
\end{equation}
The following result establishes equivalence between the Euler--Lagrange equations for $S_{\gamma}$ and
the second-order differential equation (\ref{second-order}). The result is proved by
a non-trivial adaptation of the variational calculus in the space of $2\pi$-periodic functions with zero mean.

\begin{lemma}
\label{lemma-critical-points}
The set of critical points of the energy functional $S_{\gamma}(u)$ in the space $L^2_{\rm per, zero} \cap L^{\infty}_{\rm per}$
such that $u(z) < \gamma$ for every $z \in \mathbb{R}$ is equivalent to the set of $2\pi$-periodic
smooth solutions of the differential equation (\ref{second-order}),
where $L^2_{\rm per, zero}$ denotes the space of square-integrable, $2\pi$-periodic, and zero-mean functions.
\end{lemma}

\begin{proof}
Let us consider the function $U$ and its perturbation $v$
in the space $L^2_{\rm per, zero}$. We represent these functions
in the Fourier form with zero mean
\begin{eqnarray}
U(z) = \sum_{n \in \mathbb{N}} A_n \cos(nz) + B_n \sin(nz),
\label{Usum}
\end{eqnarray}
and
\begin{eqnarray*}
v(z) = \sum_{n \in \mathbb{N}} a_n \cos(nz) + b_n \sin(nz).
\end{eqnarray*}
Similarly, we expand
$$
U^2(z) = C_0 + \sum_{n \in \mathbb{N}} C_n \cos(nz) + D_n \sin(nz),
$$
where $C_0 \neq 0$. Expanding $S_{\gamma}(U+v) - S_{\gamma}(U)$ to the linear order in $v$ and using
Parseval's equality, we obtain the first variation of $S_{\gamma}$, denoted by $\delta S_{\gamma}$, in
the explicit form
$$
\delta S_{\gamma} = \pi \sum_{n \in \mathbb{N}} a_n \left[ 2 n^{-2} A_n - 2 \gamma A_n + C_n \right]
+ b_n \left[ 2 n^{-2} B_n - 2 \gamma B_n + D_n \right].
$$
Although $C_0 \neq 0$, it does not enter to the first variation $\delta S_{\gamma}$. The Euler--Lagrange equations
for vanishing first variation in the Fourier coefficients take the form
\begin{equation}
\label{first-variation}
 2 n^{-2} A_n - 2 \gamma A_n + C_n  = 0, \quad 2 n^{-2} B_n - 2 \gamma B_n + D_n = 0, \quad n \in \mathbb{N}.
\end{equation}
Note that these equations are not equivalent to the integral equation
\begin{equation*}
U( U - 2\gamma) - 2 \partial_z^{-2} U = 0,
\end{equation*}
because the mean value of $U^2$ is nonzero. They are, however, equivalent to the equation
\begin{equation}
\label{double-integration}
P_0 U^2 - 2\gamma U - 2 \partial_z^{-2} U = 0,
\end{equation}
where $\partial_z^{-2}$ denotes the zero-mean second anti-derivative of the $2\pi$-periodic function $U$
and $P_0$ is the projection operator from $L^2_{\rm per}$ to $L^2_{\rm per, zero}$.

Multiplying (\ref{first-variation}) by $\frac{n}{2}$ for $n \in \mathbb{N}$
and taking the inverse Fourier transform, we obtain
\begin{equation}
\label{first-integration}
(U - \gamma) U'(z) - \partial_z^{-1} U = 0,
\end{equation}
which is well-defined in $L^2_{\rm per, zero}$ if $U \in L^2_{\rm per, zero}$.
Equation (\ref{first-integration}) is nothing but the derivative of (\ref{double-integration}) in $z$.
By bootstrapping arguments, if $U(z) < \gamma$ for every $z \in \mathbb{R}$,
then the solution $U$ is actually smooth in $z$.
Hence, equation (\ref{first-integration}) can be differentiated with respect to $z$, after which
it becomes the second-order differential equation (\ref{second-order}). Thus, the equivalence is
established in one direction.

In the opposite direction, if $U$ is a $2\pi$-periodic smooth solution of the differential equation
(\ref{second-order}), then it has zero mean and belongs to the space $L^2_{\rm per,zero} \cap L^{\infty}_{\rm per}$.
Then, working backwards from (\ref{second-order}) to (\ref{first-integration}), (\ref{double-integration}), and
(\ref{first-variation}), we obtain $\delta S_{\gamma} = 0$ for every perturbation in $L^2_{\rm per,zero}$.
\end{proof}

Let us now consider the other conserved quantities of the reduced Ostrovsky equation (\ref{redOst}) \cite{BS,GP}.
If $u \in H^3$ with $1 - 3 u_{xx} \geq F_0 >  0$ for every $x \in \mathbb{R}$ and some constant $F_0$,
then the reduced Ostrovsky equation (\ref{redOst}) has conserved higher-order energy given by
\begin{equation}
\label{higher-order-energy-2}
H(u) = \int \frac{(u_{xxx})^2}{(1 - 3 u_{xx})^{7/3}} dx,
\end{equation}
as well as the conserved Casimir-type functional
\begin{equation}
\label{higher-order-energy-1}
C(u) = \int (1 - 3 u_{xx})^{1/3} dx.
\end{equation}
The conserved quantity $C(u)$ is a mass integral associated to the
quantity $f := (1 - 3 u_{xx})^{1/3}$, which determines
the Jacobian of the coordinate transformation that brings the reduced
Ostrovsky equation to the integrable Klein--Gordon-type equation \cite{BS,GH,GP}.

Let us define the second energy functional in the form
\begin{equation}
\label{second-Lyapunov}
R_{\Gamma}(u) := C(u) - \Gamma Q(u),
\end{equation}
where parameter $\Gamma$ should be chosen in such a way that
the same periodic wave $U$ given by solution of the second-order
equation (\ref{second-order}) becomes a critical point of $R_{\Gamma}(u)$
in function space $H^s_{\rm per}$ with $s > \frac{5}{2}$ (so that $u_{zz}$
is a bounded and continuous function by Sobolev's embedding).
The following lemma specified the proper definition of $\Gamma$.

\begin{lemma}
\label{lemma-higher-order-energy}
The  set of $2\pi$-periodic smooth solutions $U$ of the differential equation (\ref{second-order})
such that $U(z) < \gamma$ and $1 - 3 U''(z) > 0$ for every $z \in \mathbb{R}$
yields critical points of the energy functional $R_{\Gamma}(u)$ in the space $H^s_{\rm per}$ with $s > \frac{5}{2}$
if and only if
\begin{equation}
\label{Gamma-defined}
\Gamma := \frac{-1}{(\gamma^3 - 6I)^{2/3}},
\end{equation}
where $\gamma^3 - 6I > 0$.
\end{lemma}

\begin{proof}
The Euler--Lagrange equation for $R_{\Gamma}(u)$ is given by the fourth-order
differential equation
\begin{equation}\label{fourth-order}
(1 - 3 u'') u'''' + 5 (u''')^2 + \Gamma u (1 - 3 u'')^{8/3} = 0.
\end{equation}
Solutions to equation (\ref{fourth-order}) are smooth in $z$
as long as $1 - 3 u''(z) > 0$ for every $z \in \mathbb{R}$.
Differentiating (\ref{second-order}) for smooth solutions $U$, we obtain
\begin{equation}
\label{derivative-1}
(\gamma - U) U''' = -U'(1 - 3 U'')
\end{equation}
and
\begin{equation}
\label{derivative-2}
(\gamma - U) U'''' = -U''(1-3U'') + 4 U' U'''.
\end{equation}
Also recall from (\ref{identity}) that if $U(z) < \gamma$ and $1 - 3 U''(z) > 0$ for every $z \in \mathbb{R}$,
then $\gamma^3 - 6I > 0$.
Substituting (\ref{identity}), (\ref{derivative-1}), and (\ref{derivative-2}) to equation (\ref{fourth-order}) for $u = U$, we
obtain the identity if and only if $\Gamma$ is given by (\ref{Gamma-defined}).
\end{proof}

\subsection{Second variation of the energy functionals}

The $2\pi$-periodic solution $U$ of the differential equation (\ref{second-order})
is a critical point of either the standard energy functional
$S_{\gamma}(u)$ or the alternative energy functional $R_{\Gamma}(u)$ given by (\ref{first-Lyapunov}) or (\ref{second-Lyapunov})
respectively, see Lemmas \ref{lemma-small-amplitude}, \ref{lemma-critical-points}, and \ref{lemma-higher-order-energy}. In order to
study extremal properties of the energy functionals at $U$, we shall study their second
variation at $U$ in appropriate function spaces.

Using a perturbation $v \in L^2_{\rm per,zero}$ to the periodic wave $U$
and expanding $S_{\gamma}(U+v) - S_{\gamma}(U)$  to the quadratic order in $v$, we
obtain the second variation of $S_{\gamma}$, denoted by $\delta^2 S_{\gamma}$, in the explicit form
\begin{equation}
\label{delta-2-S}
\delta^2 S_{\gamma} = \int \left[ (\partial_z^{-1} v)^2 - (\gamma - U) v^2 \right] dz.
\end{equation}
The domain of the integral depends on the properties of the perturbation $v$.
In what follows, we assume that $v$ is periodic with the period $2 \pi N$ for an integer $N$,
so that $L^2_{\rm per,zero}$ denotes now the space of square integrable $2\pi N$-periodic functions
with the zero mean value.

Since the first term of $\delta^2 S_{\gamma}$ is positive and the second term is negative,
the sign of $\delta^2 S_{\gamma}$ is not generally defined. In order to characterize $\delta^2 S_{\gamma}$,
we write (\ref{delta-2-S}) as the quadratic form
\begin{equation}
\label{operator-L}
\delta^2 S_{\gamma} = \langle L_{\gamma} v, v \rangle_{L^2},
\end{equation}
where
\begin{equation}
\label{operator-L-explicit}
L_{\gamma} := -\partial_z^{-2} - \gamma + U : \; L^2_{\rm per,zero} \to L^2_{\rm per,zero}.
\end{equation}
Note that the range of $L_{\gamma}$ is defined in $L^2_{\rm per, zero}$ by using the projection
operator $P_0$ from the proof of Lemma \ref{lemma-critical-points}.
The following lemma characterizes the spectrum of the self-adjoint operator $L_{\gamma}$
in space $L^2_{\rm per,zero}$ given by (\ref{operator-L-explicit}).

\begin{lemma}
\label{lemma-spectrum-S}
For every integer $N \geq 1$ and every $\gamma > 1$ such that $|\gamma - 1|$ is sufficiently small,
the spectrum of $L_{\gamma}$ in $L^2_{\rm per,zero}$
consists of $2N-1$ positive eigenvalues (counted by their multiplicity),
a simple zero eigenvalue, and an infinite number of negative eigenvalues.
\end{lemma}

\begin{proof}
For $a = 0$, $\gamma = 1$ and $U(z) = 0$, the spectrum of $L_{\gamma = 1}$
consists of a sequence of double eigenvalues
\begin{equation}
\label{spectrum-limiting}
a = 0 : \quad \sigma(L_{\gamma = 1}) = \left\{ -1 + \frac{N^2}{n^2}, \quad n \in \mathbb{N} \right\},
\end{equation}
as follows from Fourier series solutions. If $N = 1$, all double eigenvalues but one are negative
and the only remaining double eigenvalue is zero.

Note that $L_{\gamma = 1} + 1 = -\partial_z^{-2} : L^2_{\rm per,zero} \to L^2_{\rm per,zero}$
is a compact operator. If $a \neq 0$ but small, the nonzero eigenvalues split generally
but remain in the negative domain, as it follows from the perturbation theory for compact operators
under bounded perturbations. In particular, if $A(\epsilon)$ is an analytic family of self-adjoint operators
and if the spectrum of $A(0)$ is separated into two parts, then this remains true also for $A(\epsilon)$ with $\epsilon$
sufficiently small (see Theorem VII.1.7 in \cite{Kato}).

The zero eigenvalue splits generally if $\gamma \neq 1$, but one eigenvalue remains
at zero because of the translational invariance of the differential equation (\ref{second-order})  resulting
in the relation $L_{\gamma} U' = 0$ that holds for every $\gamma \in \left(1, \frac{\pi^2}{9}\right)$. Therefore, we only need to
check what happens with the other eigenvalue bifurcating from zero if $a \neq 0$.
This is done with the regular perturbation expansions.

We expand the operator $L_{\gamma}$ in powers of $a$:
$$
L_{\gamma} = -\partial_z^{-2} - 1 + a \cos(z) + \frac{a^2}{6} (2 \cos(2z) - 1) + \mathcal{O}_{L^2_{\rm per,zero} \to L^2_{\rm per,zero}}(a^3).
$$
Since $U'(z) = -a \sin(z) + \mathcal{O}(a^2)$ persists as the eigenfunction of $L_{\gamma}$ for zero eigenvalue,
we are looking for the perturbation expansion starting with the eigenfunction $V_0 = \cos(z)$:
\begin{equation}
\label{asympt-V-1}
\lambda = a \lambda_1 + a^2 \lambda_2 + \mathcal{O}(a^3), \quad
V = V_0 + a V_1 + a^2 V_2 + \mathcal{O}_{L^2_{\rm per,zero}}(a^3),
\end{equation}
where corrections $V_{1,2}$ are uniquely defined under
the orthogonality constraints $\langle V_0, V_{1,2} \rangle_{L^2} = 0$.
At the linear order in $a$, we obtain
$$
(1 + \partial_{z}^{-2}) V_1 = P_0 \left(\cos^2(z) - \lambda_1 \cos(z)\right) =
\frac{1}{2} \cos(2z) - \lambda_1 \cos(z),
$$
where $P_0$ is the projection operator from $L^2_{\rm per}$ to $L^2_{\rm per,zero}$.
Fredholm's solvability condition and orthogonality for $V_1$ yields the unique solution
\begin{equation}
\label{lambda-1-solution}
\lambda_1 = 0, \quad V_1(z) = \frac{2}{3} \cos(2z).
\end{equation}
At the order of $a^2$, we obtain
$$
(1 + \partial_{z}^{-2}) V_2 = \cos(z) \cos(2z) - \frac{1}{6} \cos(z) - \lambda_2 \cos(z).
$$
Fredholm's solvability condition yields
\begin{equation}
\label{lambda-2-S}
\lambda_2 = \frac{1}{3},
\end{equation}
hence, the zero eigenvalue associated with the even eigenfunction $V_0$ for $a = 0$
becomes positive eigenvalue for nonzero but small values of $a$.
This proves the assertion of the lemma for $N = 1$.

From the explicit representation (\ref{spectrum-limiting}) for $a = 0$, we realize
that there are additional $2(N-1)$ positive eigenvalues of $L_{\gamma = 1}$ if $N \geq 2$ (counting
with their multiplicity). These positive eigenvalues remain positive for nonzero but small values
of $a$. In addition, the perturbation expansions (\ref{asympt-V-1}), (\ref{lambda-1-solution}), and (\ref{lambda-2-S}).
show that an additional positive eigenvalue
of $L_{\gamma}$ bifurcates from zero for any $N \geq 1$. The assertion of the lemma is proved.
\end{proof}

\begin{remark}
The result of Lemma \ref{lemma-spectrum-S} is typical for periodic waves.
For $N = 1$, the positive eigenvalue can be removed under the constraint
that the perturbation $v$ does not change the momentum functional $Q(u) = \| u \|^2_{L^2}$. As a result,
one can use $-S_{\gamma}$ as a Lyapunov functional in the proof of orbital stability
of the periodic waves with respect to the perturbations of the same period \cite{Pava}.
The constraint is used to specify a varying wave speed as a function of time.
In addition, one needs to introduce a varying parameter of spatial translation as a function of time
to remove the degeneracy of the quadratic form at the zero eigenvalue of $L_{\gamma}$.

On the other hand, for $N \geq 2$, the additional positive eigenvalues of $L_{\gamma}$ cannot be removed
by one constraint on the momentum functional $Q$, so that $-S_{\gamma}$ can not
be chosen as the Lyapunov functional in the proof of orbital stability
of the periodic waves with respect to the subharmonic perturbations.
\end{remark}

Let us now inspect the second variation of the alternative
energy functional $R_{\Gamma}(u)$ given by (\ref{second-Lyapunov}).
Using a perturbation $v \in H^2_{\rm per,zero}$ to the periodic wave $U$
and expanding $R_{\Gamma}(U+v) - R_{\Gamma}(U)$ to the quadratic order in $v$,
we obtain
\begin{equation}
\label{delta-2-R}
\delta^2 R_{\Gamma} := \int \left[ \frac{v^2}{(\gamma^3 - 6 I)^{2/3}} - \frac{v_{zz}^2}{(1-3U'')^{5/3}} \right] dz.
\end{equation}
Note that $1 - 3 U''(z) > 0$ for every $z$
and $\gamma^3 - 6I > 0$ if $\gamma \in \left(1, \frac{\pi^2}{9}\right)$.
The first term of $\delta^2 R_{\Gamma}$ is again positive and the second term is negative,
so that the sign of $\delta^2 R_{\Gamma}$ is not generally defined. In order to characterize
$\delta^2 R_{\Gamma}$, we write (\ref{delta-2-R})  as the quadratic form
\begin{equation}
\label{operator-M}
\delta^2 R_{\Gamma} = \langle M_{\gamma} v, v \rangle_{L^2},
\end{equation}
where
\begin{equation}
\label{operator-M-explicit}
M_{\gamma} := -\partial_z^2 (1-3 U'')^{-5/3} \partial_z^2 + (\gamma^3 - 6I)^{-2/3} : H^4_{\rm per,zero} \to L^2_{\rm per,zero}.
\end{equation}
The following lemma characterizes the spectrum of $M_{\gamma}$ given by (\ref{operator-M-explicit}).

\begin{lemma}
\label{lemma-spectrum-R}
For every integer $N \geq 1$ and every $\gamma > 1$ such that $|\gamma - 1|$ is
sufficiently small, the spectrum of $M_{\gamma}$ in $L^2_{\rm per,zero}$
consists of $2N-2$ positive eigenvalues (counted by their multiplicity),
a simple zero eigenvalue, and an infinite number of negative eigenvalues.
\end{lemma}

\begin{proof}
For $a = 0$, $\gamma = 1$, $I = 0$, and $U(z) = 0$, the spectrum of $M_{\gamma = 1}$ in $L^2_{\rm per,zero}$
consists of a sequence of double eigenvalues
\begin{equation}
\label{spectrum-limiting-M}
a = 0 : \quad \sigma(M_{\gamma = 1}) = \left\{ 1 - \frac{n^4}{N^4}, \quad n \in \mathbb{N} \right\}.
\end{equation}
If $N = 1$, all double eigenvalues but one are negative
and the only remaining eigenvalue is zero. The zero eigenvalue splits again, but one eigenvalue remains
at zero because of the exact relation $M_{\gamma} U' = 0$ that holds for every $\gamma \in \left(1, \frac{\pi^2}{9}\right)$. Indeed,
using (\ref{identity}), we check that $M_{\gamma} U' = 0$ is equivalent to the differential equation
$$
-\partial_z^2 (\gamma - U)^5 U''' + (\gamma^3 - 6I) U' = 0.
$$
Integrating once with zero mean and using (\ref{derivative-1}), we obtain the differential equation
$$
\partial_z (\gamma - U)^4 (1 - 3 U'') U' + (\gamma^3 - 6I) U = 0,
$$
which is equivalent to the differential equation (\ref{second-order}) due to the identity (\ref{identity}).

Let us now show that the zero eigenvalue of $M_{\gamma = 1}$
associated with the even function $V_0(z) = \cos(z)$ at $a = 0$
becomes negative for nonzero but small values of $a$. This is done by the regular perturbation expansions.
We expand the operator $M_{\gamma}$ in powers of $a$:
$$
M_{\gamma} = -\partial_z^{4} + 1 + 5 a \partial_z^2 \cos(z) \partial_z^2
+ \frac{5a^2}{3} - 10 a^2 \partial_z^4 -
\frac{10 a^2}{3} \partial_z^2 \cos(2z) \partial_z^2 + \mathcal{O}_{H^4_{\rm per,zero} \to L^2_{\rm per,zero}}(a^3).
$$
We are looking for the perturbation expansion starting with the eigenfunction $V_0 = \cos(z)$:
\begin{equation}
\label{asympt-V-1-R}
\lambda = a \lambda_1 + a^2 \lambda_2 + \mathcal{O}(a^3), \quad
V = V_0 + a V_1 + a^2 V_2 + \mathcal{O}_{L^2_{\rm per,zero}}(a^3),
\end{equation}
where corrections $V_{1,2}$ are uniquely defined under
the orthogonality constraints $\langle V_0, V_{1,2} \rangle_{L^2} = 0$.
At the linear order in $a$, we obtain
$$
(1 - \partial_{z}^4) V_1 = - 5 P_0 \left( \partial_z^2 \cos(z) \partial_z^2 \cos(z) \right)+ \lambda_1 \cos(z) =
-10 \cos(2z) + \lambda_1 \cos(z).
$$
Fredholm's solvability condition and orthogonality for $V_1$ yields the unique solution
\begin{equation}
\label{lambda-1-solution-M}
\lambda_1 = 0, \quad V_1(z) = \frac{2}{3} \cos(2z).
\end{equation}
At the order of $a^2$, we obtain
$$
(1 - \partial_{z}^4) V_2 = -\frac{10}{3} \partial_z^2 \cos(z) \partial_z^2 \cos(2z)
+ \frac{10}{3} \partial_z^2 \cos(2z) \partial_z^2 \cos(z) + \frac{25}{3} \cos(z) + \lambda_2 \cos(z).
$$
Fredholm's solvability condition yields
\begin{equation}
\label{lambda-2-R}
\lambda_2 = -\frac{10}{3},
\end{equation}
hence, the assertion of the lemma is proved for $N = 1$. The proof for $N \geq 2$ is similar to
that of Lemma \ref{lemma-spectrum-S}.
\end{proof}

\begin{remark}
Compared to the functional $-S_{\gamma}$, periodic waves are unconstrained minimizers of the functional $-R_{\Gamma}$
with respect to perturbations of the same period. Therefore, $-R_{\Gamma}$ can be used as a Lyapunov
functional in the proof of orbital stability of periodic waves and the only varying parameter
of spatial translation is needed to remove the degeneracy of the quadratic form for
the zero eigenvalue of $M_{\gamma}$, see, e.g., the proof of Theorem 1.8 in \cite{GP1}.
\end{remark}

\subsection{Positivity of a linear combination of the second variations $\delta^2 S_{\gamma}$ and $\delta^2 R_{\Gamma}$}

From Lemmas \ref{lemma-spectrum-S} and \ref{lemma-spectrum-R}, we can see that
properties of the operator $M_{\gamma}$ are similar to those of the operator $L_{\gamma}$ with the exception
of one small eigenvalue, which is positive for $L_{\gamma}$ and negative for $M_{\gamma}$.
Let us define the Lyapunov functional by a linear combination
of the two energy functionals $S_{\gamma}(u)$ and $R_{\Gamma}(u)$ given by (\ref{first-Lyapunov}) and (\ref{second-Lyapunov})
respectively:
\begin{equation}
\label{Lyapunov}
\Lambda_{c,\gamma}(u) := S_{\gamma}(u) - c R_{\Gamma}(u),
\end{equation}
where $c \in \mathbb{R}$ is a parameter to be defined within an appropriate interval.
The energy function $\Lambda_{c,\gamma}(u)$ is defined for $u \in H^s_{\rm per, zero}$ with $s > \frac{5}{2}$.
By Lemmas \ref{lemma-critical-points} and  \ref{lemma-higher-order-energy},
the $2\pi$-periodic smooth solution $U$ of the differential equation (\ref{second-order})
is a critical point of $\Lambda_{c,\gamma}$ for every $c \in \mathbb{R}$.
The second variation of $\Lambda_{c,\gamma}$, denoted by $\delta^2 \Lambda_{c,\gamma}$,
is defined for the perturbation $v$ in $H^2_{\rm per, zero}$, which may have period $2 \pi N$
for an integer $N$.

The following result shows that the second variation $\delta^2 \Lambda_{c,\gamma = 1}$
is positive for a particular value of the parameter $c$.

\begin{lemma}
\label{lemma-positivity}
$\delta^2 \Lambda_{c,\gamma = 1} \geq 0$ for every $v \in H^2_{\rm per, zero}$ and every $N \geq 1$ if and only if $c = \frac{1}{2}$.
\end{lemma}

\begin{proof}
We substitute $a = 0$, $\gamma = 1$, $I = 0$, and $U(z) = 0$ in (\ref{delta-2-S}) and (\ref{delta-2-R})
and obtain
\begin{equation}
\label{Lyapunov-zero}
a = 0 : \quad \delta^2 \Lambda_{c,\gamma=1} = \int \left[ (\partial_z^{-1} v)^2 - (1+c) v^2 + c (\partial_z^2 v)^2 \right] dz.
\end{equation}
The second variation $\delta^2 \Lambda_{c,\gamma=1}$ is well-defined for $v \in H^2_{\rm per, zero}$, e.g.,
in the space of $2\pi N$-periodic functions with zero mean for an integer $N$. To deal with every integer $N$,
we are looking for the values of $c$, for which $\delta^2 \Lambda_{c,\gamma=1} \geq 0$
for every $v \in H^2(\mathbb{R}) \cap \dot{H}^{-1}(\mathbb{R})$.

By the Fourier transform on the line, the problem of determining if $\delta^2 \Lambda_{c,\gamma=1} \geq 0$ in (\ref{Lyapunov-zero})
is identical to the problem of finding $c$ for which
$$
D_c(k) := ck^4 - (1+c) + k^{-2} \geq 0, \quad k \in \mathbb{R}.
$$
The values $c = 0$ and $c < 0$ do not provide positivity of $D_c(k)$ as $|k| \to \infty$. Therefore, we consider $c > 0$ only.

Since $D_c'(k) = 4 c k^3 - 2 k^{-3}$, the critical points of $D_c$ exist if $c > 0$ and correspond to $k_c^6 = (2c)^{-1}$.
Since $D_c''(k) = 12 c k^2 + 6 k^{-4} > 0$, the critical points
of $D_c$ at $k = \pm k_c$ are the points of minimum of $D_c$. Therefore, we can compute
$$
F(c) := D_c(k_c) = c k_c^4 - (1+c) + k_c^{-2} =  \frac{3}{2^{2/3}} c^{1/3} -1 - c.
$$
We are looking for the values of $c$ for which $F(c) \geq 0$, so that $D_c(k) \geq 0$ for all $k \in \mathbb{R}$.
Since $F'(c) = (2c)^{-2/3} - 1$, the only critical point of $F$ occurs at $c_0 = \frac{1}{2}$.
Since $F''(c) < 0$, $c_0$ is the point of maximum of $F$,
for which we have $F(c_0) = 0$. Therefore, $F(c) \leq 0$ and the only possible value of $c$
for which $D_c(k) \geq 0$ for every $k \in \mathbb{R}$ is the value $c = c_0 = \frac{1}{2}$. Indeed,
in this case, we can factorize the dispersion relation in the form
\begin{equation}
\label{factorization}
D_{c = c_0}(k) = \frac{1}{2} k^4 - \frac{3}{2} + \frac{1}{k^2} = \frac{(1-k^2)^2 (2 + k^2)}{2 k^2} \geq 0.
\end{equation}
The statement of the lemma is proved.
\end{proof}

By a perturbation argument, we shall verify that $\delta^2 \Lambda_{c,\gamma}$
remains nonnegative for every $\gamma$ such that $|\gamma - 1|$ is sufficiently small,
provided the parameter $c$ is close enough to $c_0$. Although
$\delta^2 \Lambda_{c,\gamma}$ still vanishes for $v(z) = U'(z)$ for every $\gamma \in \left(1,\frac{\pi^2}{9}\right)$
due to the translational symmetry, we will show that the zero eigenvalue is simple
and $\delta^2 \Lambda_{c,\gamma} > 0$ if $v \in H^2_{\rm per, zero}$ is nonzero and orthogonal to $U'$, 
provided that $|\gamma - 1|$ and $|c-c_0|$ are sufficiently small.
The following theorem represents the main result of this paper
for the reduced Ostrovsky equation (\ref{redOst}). 

Since $1 - 3 U''(z) > 0$ for every $z \in \mathbb{R}$ if $\gamma \in \left(1,\frac{\pi^2}{9}\right)$,
a global solution to the reduced Ostrovsky equation (\ref{redOst}) exists for sufficiently small
$v \in H^s_{\rm per,zero}$ with $s > \frac{5}{2}$ \cite{GP,LPS1}. The value of $\Lambda_{c,\gamma}$
conserves in the time evolution of the reduced Ostrovsky equation (\ref{redOst}). 
Orbital stability of the periodic wave with profile $U$ with respect to subharmonic perturbations
$v \in H^s_{\rm per, zero}$ with $s > \frac{5}{2}$  follows from
the positivity of $\delta^2 \Lambda_{c,\gamma}$ \cite{GP1}. The choice of $H^s_{\rm per,zero}$ with $s > \frac{5}{2}$
is explained by the necessity to control $v_{zz} \in L^{\infty}_{\rm per}$ by Sobolev's embedding.

\begin{theorem}
\label{theorem-red-Ost}
There exists $\gamma_0 \in \left(1,\frac{\pi^2}{9}\right)$ and $C > 0$ such that
$\delta^2 \Lambda_{c,\gamma} \geq C \| v \|_{H^2_{\rm per}}^2$ for every $\gamma \in (1,\gamma_0)$ and
every $v \in H^2_{\rm per, zero}$ such that $\langle U', v \rangle_{L^2_{\rm per}} = 0$,
if $c \in (c_-,c_+)$, where $c_{\pm}$ are given by the asymptotic expansion
\begin{equation}
\label{exact-c-plus-minus}
c_{\pm} = \frac{1}{2} \pm \frac{\sqrt{3} a}{2} + \mathcal{O}(a^2), \quad \mbox{\rm as} \quad a \to 0,
\end{equation}
where $a$ determines $\gamma$ and $I$ in the Stokes expansions (\ref{gamma}) and (\ref{E-parameter}).
\end{theorem}

\begin{remark}
Using numerical computations, we will show in Section 5 that the result of Theorem \ref{theorem-red-Ost}
extends to every $\gamma$ in the interval $\left(1,\frac{\pi^2}{9}\right)$. The boundaries $c_{\pm}$
will be approximated numerically. It is quite possible that analytical expressions for $c_{\pm}$
exist, but obtaining analytical expressions requires a heavy use of integrability
features of the reduced Ostrovsky equation (\ref{redOst}).
\end{remark}

\section{Main result for the modified reduced Ostrovsky equation (\ref{redModOst})}

Here we obtain similar results for the modified reduced Ostrovsky equation (\ref{redModOst}).
For the sake of brevity, we only outline the results and skip proofs which follow those in
Section 2.

\subsection{Travelling waves}

Travelling $2L$-periodic solutions of the modified reduced Ostrovsky equation (\ref{redModOst})
can be expressed in the normalized form
\begin{equation}
\label{trav-wave-mod}
u(x,t) = \frac{L}{\pi} U(z), \quad z = \frac{\pi}{L} x - \frac{L}{\pi} \gamma t,
\end{equation}
where $U(z)$ is a $2\pi$-periodic, zero-mean solution of the second-order differential equation
\begin{equation}
\label{second-order-mod}
\frac{d}{dz} \left[ \left( \gamma - \frac{1}{2} U^2 \right) \frac{dU}{dz} \right] + U(z) = 0,
\end{equation}
and the parameter $\gamma$ is proportional to the wave speed.
The second-order equation (\ref{second-order-mod}) can be integrated with the first-order invariant
\begin{equation}
\label{first-order-mod}
I = \frac{1}{2} \left( \gamma - \frac{1}{2} U^2 \right)^2
\left( \frac{dU}{dz} \right)^2 + \frac{\gamma}{2} U^2 - \frac{1}{8} U^4 = {\rm const}.
\end{equation}
By the translational invariance and reversibility of the differential equation (\ref{second-order-mod}),
$U$ is selected to be an even function of $z$. The following lemma reports an analogue of Lemma \ref{lemma-small-amplitude}.
The proof is similar, so it will be omitted.

\begin{lemma}
\label{lemma-small-amplitude-mod}
For every $\gamma > 1$ such that $|\gamma - 1|$ is sufficiently small,
there exists a unique $2\pi$-periodic, even, smooth solution $U$,
which can be parameterized by the amplitude parameter $a$, such that
\begin{equation}
\label{Stokes-mod}
U(z) = a \cos(z) + \frac{3}{64} a^3 \cos(3z) + \mathcal{O}_{C^{\infty}_{\rm per}}(a^5)
\quad \mbox{\rm as} \quad a \to 0,
\end{equation}
where parameter $a$ determines the asymptotic expansions
\begin{equation}
\label{gamma-parameter-mod}
\gamma = 1 + \frac{1}{8} a^2 + \mathcal{O}(a^4)
\end{equation}
and
\begin{equation}
\label{E-parameter-mod}
I = \frac{1}{2} a^2 + \mathcal{O}(a^4) \quad \mbox{\rm as} \quad a \to 0.
\end{equation}
\end{lemma}

\begin{remark}
The transformation
\begin{equation}
\label{wave-transformation-mod}
U(z) = {\bf u}(\zeta), \quad z = \int_0^{\zeta} \left( \gamma - \frac{1}{2} {\bf u}(\zeta')^2 \right) d \zeta'
\end{equation}
reduces(\ref{second-order-mod}) to the second-order equation 
\begin{equation}
\label{KdV-ODE-mod}
\frac{d^2 {\bf u}}{d \zeta^2} + \left( \gamma - \frac{1}{2} {\bf u}^2 \right) {\bf u} = 0.
\end{equation}
The latter equation arises for travelling waves of the modified KdV equation and it admits
explicit solutions for periodic waves given by the Jacobi elliptic {\rm sn} functions.
The family of $2\pi$-periodic waves exists in variables $U$ and $z$
as long as the transformation (\ref{wave-transformation-mod})
is invertible, that is, as long as $|{\bf u}(\zeta)| < \sqrt{2\gamma}$ for every $\zeta$.
\end{remark}

\begin{remark}
The periodic wave with profile $U$ exists for every $\gamma \in \left(1,\frac{\pi^2}{8}\right)$ \cite{JG}.
As $\gamma \to \frac{\pi^2}{8}$,
the periodic wave is no longer smooth at the end points of the fundamental period.
The limiting wave has the piecewise linear profile:
\begin{equation}
\label{linear-wave}
\gamma = \frac{\pi^2}{8} : \quad U(z) = |z| - \frac{\pi}{2},
\end{equation}
such that $U(\pm \pi) = \sqrt{2 \gamma}$. As $\gamma \to \frac{\pi^2}{8}$,
the piecewise linear wave profile (\ref{linear-wave}) yields the value
\begin{equation}
\label{first-order-final-mod}
I = \frac{1}{2} \gamma^2 = \frac{\pi^4}{128}.
\end{equation}
From the first-order invariant (\ref{first-order-mod}), we deduce that
\begin{equation}
\label{identity-mod}
1 - \left( \frac{dU}{dz} \right)^2 = \frac{\gamma^2 - 2 I}{(\gamma - \frac{1}{2}U^2)^2}.
\end{equation}
Since $|U'(z)| < 1$ and $|U(z)| < \sqrt{2 \gamma}$ for every $z \in \mathbb{R}$
if $\gamma \in \left(1,\frac{\pi^2}{8}\right)$, we have
$\gamma^2 - 2 I > 0$ for every $\gamma \in \left(1,\frac{\pi^2}{8}\right)$ \cite{JG}.
\end{remark}

\subsection{Conserved quantities}

Local solutions $u$ of the modified reduced Ostrovsky equation (\ref{redModOst})
in the space $H^s \cap \dot{H}^{-1}$ with $s > \frac{3}{2}$ \cite{LPS2}
have conserved momentum $Q(u) = \| u \|_{L^2}^2$ and energy
\begin{equation}
\label{energy_mod}
E(u) = \| \partial_x^{-1} u \|_{L^2}^2 + \frac{1}{12} \| u \|_{L^4}^4,
\end{equation}
so that the first energy functional $S_{\gamma}(u)$ can be considered in the same form (\ref{first-Lyapunov}).
The periodic wave profile $U$ satisfying the differential equation (\ref{second-order-mod})
is a critical point of the energy functional $S_{\gamma}(u)$. The proof is similar to the one in Lemma \ref{lemma-critical-points}.

The other two conserved quantities of the modified reduced Ostrovsky equation (\ref{redModOst}) \cite{Br,JG,PelSak}
are defined if $u \in H^2$ with $1 - u_x^2 \geq F_0 >  0$ for every $x \in \mathbb{R}$ and some constant $F_0$,
The conserved higher-order energy is given by
\begin{equation}
\label{higher-order-energy-2-mod}
H(u) = \int_{\mathbb{R}}\frac{u_{xx}^{2}}{(1-u_{x}^{2})^{5/2}}dx
\end{equation}
where the conserved Casimir-type functional is given by
\begin{equation}
\label{higher-order-energy-1-mod}
C(u) = \int (1 - u_{x}^2)^{1/2} dx.
\end{equation}
The conserved quantity $C(u)$ is a mass integral associated to the
quantity $f := (1 - u_x^2)^{1/2}$, which determines
the Jacobian of the coordinate transformation that brings the reduced
modified Ostrovsky equation to the integrable Klein--Gordon-type equation \cite{Br,JG}.

Let us define the second energy functional $R_{\Gamma}(u)$ by
the same expression (\ref{second-Lyapunov}). The functional $R_{\Gamma}(u)$
is well defined in function space $H^s_{\rm per}$ with $s > \frac{3}{2}$ (so that $u_{z}$
is a bounded and continuous function by Sobolev's embedding).
The following lemma gives an analogue of Lemma \ref{lemma-higher-order-energy}.

\begin{lemma}
\label{lemma-higher-order-energy-mod}
The set of $2\pi$-periodic smooth solutions $U$
of the differential equation (\ref{second-order-mod})
such that $|U'(z)| < 1$ and $|U(z)| < \sqrt{2 \gamma}$ for every $z \in \mathbb{R}$
yields critical points of the energy functional $R_{\Gamma}(u)$ given by (\ref{second-Lyapunov})
in the space $H^s_{\rm per}$ with $s > \frac{3}{2}$ if and only if
\begin{equation}
\label{Gamma-defined-mod}
\Gamma := \frac{-1}{2 (\gamma^2 - 2I)^{1/2}},
\end{equation}
where $\gamma^2 - 2I > 0$.
\end{lemma}

\begin{proof}
The Euler--Lagrange equation for the second energy functional $R_{\Gamma}(u)$ yields
another second-order differential equation
\begin{equation}\label{fourth-order-mod}
\frac{d}{dz} \left[ \frac{u'}{\sqrt{1 - (u')^2}} \right] - 2 \Gamma u = 0 \quad \Rightarrow \quad
\frac{d^2 u}{d z^2} = 2 \Gamma u (1 - (u')^2)^{3/2}.
\end{equation}
Substituting the second and first derivatives from (\ref{second-order-mod}) and (\ref{identity-mod})
to equation (\ref{fourth-order-mod}) with $u = U$, we obtain an identity if and only if $\Gamma$
is given by (\ref{Gamma-defined-mod}).
\end{proof}

\subsection{Second variation of the energy functionals}

Adding a perturbation $v \in L^2_{\rm per,zero}$ to the periodic wave $U$ and expanding
$S_{\gamma}(U+v) - S_{\gamma}(U)$ to the quadratic order in $v$, we obtain the second
variation in the form
\begin{equation}
\label{delta-2-S-mod}
\delta^2 S_{\gamma} = \int \left[ (\partial_z^{-1} v)^2 - \left(\gamma - \frac{1}{2} U^2 \right) v^2 \right] dz.
\end{equation}
As previously, we assume that $v$ is the $2\pi N$-periodic function with zero mean,
where $N$ is a positive integer. The second variation $\delta^2 S_{\gamma}$ is sign-indefinite, because
the first term of $\delta^2 S_{\gamma}$ is positive and the second term is negative.

Similarly, adding a perturbation $v \in H^1_{\rm per, zero}$ and
expanding $R_{\Gamma}(U+v) - R_{\Gamma}(U)$ to the quadratic order in $v$, we obtain
the second variation in the form
\begin{equation}
\label{delta-2-R-mod}
\delta^2 R_{\Gamma} = \int \left[ \frac{v^2}{2 (\gamma^2 - 2 I)^{1/2}} - \frac{v_z^2}{2(1-(U')^2)^{3/2}} \right] dz.
\end{equation}
Again, the second variation $\delta^2 R_{\Gamma}$ is sign-indefinite, because
the first term of $\delta^2 R_{\Gamma}$ is positive and the second term is negative.

The second variations $\delta^2 S_{\gamma}$ and $\delta^2 R_{\Gamma}$ can be expressed as
quadratic forms (\ref{operator-L}) and (\ref{operator-M}) associated with the operators
$$
L_{\gamma} := -\partial_z^{-2} - \gamma + \frac{1}{2} U^2 : \; L^2_{\rm per,zero} \to L^2_{\rm per,zero}.
$$
and
$$
M_{\gamma} := \frac{1}{2} (\gamma^2 - 2 I)^{-1/2} + \frac{1}{2} \partial_z (1-(U')^2)^{-3/2} \partial_z : \;
H^2_{\rm per, zero} \to L^2_{\rm per, zero}.
$$
The following two lemmas report analogues of Lemmas \ref{lemma-spectrum-S} and \ref{lemma-spectrum-R}.

\begin{lemma}
\label{lemma-spectrum-S-mod}
For every integer $N \geq 1$ and every $\gamma > 1$ such that $|\gamma - 1|$ is
sufficiently small, the spectrum of $L_{\gamma}$ in $L^2_{\rm per,zero}$
consists of $2N-1$ positive eigenvalues (counted by their multiplicity),
a simple zero eigenvalue, and an infinite number of negative eigenvalues.
\end{lemma}

\begin{proof}
For $a = 0$, $\gamma = 1$ and $U(z) = 0$, the spectrum of $L_{\gamma = 1}$ is given by the same
formula (\ref{spectrum-limiting}) as in the proof of Lemma \ref{lemma-spectrum-S}.
Therefore, it is only necessary to develop the
regular perturbation theory for the splitting of the double zero eigenvalue as $a \neq 0$.
We expand the operator $L_{\gamma}$ in powers of $a$:
$$
L_{\gamma} = -\partial_z^{-2} - 1 + \frac{a^2}{8} (4 \cos^2(z) - 1) + \mathcal{O}_{L^2_{\rm per,zero} \to L^2_{\rm per,zero}}(a^4).
$$
Since the $\mathcal{O}_{L^2_{\rm per,zero} \to L^2_{\rm per,zero}}(a)$ term is absent in the expansion
of $L_{\gamma}$, the perturbation expansion starting with the eigenfunction $V_0 = \cos(z)$ is shorter
than in the case of Lemma \ref{lemma-spectrum-S}:
\begin{equation}
\label{expansion-same}
\lambda = a^2 \lambda_2 + \mathcal{O}(a^4), \quad V = V_0 + a^2 V_2 + \mathcal{O}_{L^2_{\rm per,zero}}(a^4),
\end{equation}
where the linear inhomogeneous equation at the order of $a^2$ bears the form
$$
(1 + \partial_{z}^{-2}) V_2 = \frac{1}{2} \cos^3(z) - \left(\frac{1}{8} + \lambda_2\right) \cos(z).
$$
Fredholm's solvability condition yields
\begin{equation}
\label{lambda-2-S-mod}
\lambda_2 = \frac{1}{4},
\end{equation}
hence, the zero eigenvalue of $L_{\gamma = 1}$
associated with the even eigenfunction $V_0$ for $a = 0$
becomes a positive eigenvalue of $L_{\gamma}$ for nonzero but small values of $a$.
\end{proof}

\begin{lemma}
\label{lemma-spectrum-R-mod}
For every integer $N \geq 1$ and every $\gamma > 1$ such that $|\gamma - 1|$ is
sufficiently small, the spectrum of $M_{\gamma}$ in $L^2_{\rm per,zero}$
consists of $2N-2$ positive eigenvalues (counted by their multiplicity),
a simple zero eigenvalue, and an infinite number of negative eigenvalues.
\end{lemma}

\begin{proof}
For $a = 0$, $\gamma = 1$, $I = 0$, and $U(z) = 0$, the spectrum of $M_{\gamma = 1}$ is given by
\begin{equation}
\label{spectrum-limiting-M-mod}
a = 0 : \quad \sigma(M_{\gamma = 1}) = \left\{ \frac{1}{2} \left(1  - \frac{n^2}{N^2} \right), \quad n \in \mathbb{N} \right\}.
\end{equation}
The zero is again a double eigenvalue, whereas the spectrum of $M_{\gamma = 1}$ is similar
to the one given by (\ref{spectrum-limiting-M}).
Therefore, we again apply the regular perturbation theory to check the splitting of the
double zero eigenvalue. We expand the operator $M_{\gamma}$ in powers of $a$:
$$
M_{\gamma} = \frac{1}{2} (1 + \partial_z^2) +
\frac{3 a^2}{16} \left[ 1 + 4 \partial_z \sin^2(z) \partial_z \right] + \mathcal{O}_{H^2_{\rm per, zero} \to L^2_{\rm per,zero}}(a^4).
$$
Using the same asymptotic expansion (\ref{expansion-same}), we obtain the linear inhomogeneous equation at $\mathcal{O}(a^2)$:
$$
\frac{1}{2} (1 + \partial_{z}^2) V_2 = -\frac{3}{4} \partial_z \sin^2(z) \partial_z \cos(z)
+ \left( \lambda_2 - \frac{3}{16} \right) \cos(z).
$$
Fredholm's solvability condition yields
\begin{equation}
\label{lambda-2-R-mod}
\lambda_2 = -\frac{3}{8},
\end{equation}
hence, the zero eigenvalue of $M_{\gamma = 1}$
associated with the even eigenfunction $V_0$ for $a = 0$
becomes a negative eigenvalue of $M_{\gamma}$ for nonzero but small values of $a$.
\end{proof}

\subsection{Positivity of a linear combination of the second variations $\delta^2 S_{\gamma}$ and $\delta^2 R_{\Gamma}$}

Let us define the Lyapunov functional $\Lambda_{c,\gamma}(u)$ by using the same linear combination
(\ref{Lyapunov}) of the energy functionals $S_{\gamma}(u)$ and $R_{\Gamma}(u)$ given by (\ref{first-Lyapunov}) and
(\ref{second-Lyapunov}), where parameter $c \in \mathbb{R}$ is to be defined within an appropriate interval.
The second variation of $\Lambda_{c,\gamma}$, denoted by $\delta^2 \Lambda_{c,\gamma}$,
is defined for the perturbation $v$ in $H^1_{\rm per, zero}$.
The following result shows that the second variation $\delta^2 \Lambda_{c,\gamma = 1}$
is positive for a particular value of the parameter $c$.

\begin{lemma}
\label{lemma-positivity-mod}
$\delta^2 \Lambda_{c,\gamma = 1} \geq 0$ for every $v \in H^1_{\rm per, zero}$ and every $N \geq 1$ 
if and only if $c = 2$.
\end{lemma}

\begin{proof}
We substitute $a = 0$, $\gamma = 1$, $I = 0$, and $U(z) = 0$ to $\delta^2 \Lambda_{c, \gamma = 1}$
and obtain
\begin{equation}
\label{Lyapunov-zero-mod}
a = 0 : \quad \delta^2 \Lambda_{c,\gamma = 1} = \int \left[ (\partial_z^{-1} v)^2 - \left( 1+\frac{c}{2} \right) v^2
+ \frac{c}{2} (\partial_z v)^2 \right] dz.
\end{equation}
Performing the Fourier transform on the line, we reduce the problem of positivity of $\delta^2 \Lambda_{c,\gamma = 1}$
to the search of values of $c$, for which 
$$
D_c(k) := \frac{c}{2} k^2 - 1 - \frac{c}{2} + k^{-2} \geq 0, \quad k \in \mathbb{R}.
$$
Again, only values $c > 0$ need to be considered.
Since $D_c'(k) = c k - 2 k^{-3}$, the critical points of $D_c$ correspond to $k_c^4 = \frac{2}{c}$ and exist if
$c > 0$. Since $D_c''(k) = c + 6 k^{-4} > 0$, the critical points of $D_c$ at $k = \pm k_c$
are the points of minimum of $D_c$. Therefore, we can compute
$$
F(c) := D_c(k_c) = \frac{c}{2} k_c^2 - 1 - \frac{c}{2} + k_c^{-2} =  (2c)^{1/2} - 1 - \frac{c}{2}.
$$
Since $F'(c) = \frac{1}{(2c)^{1/2}} - \frac{1}{2}$, the only critical point of $F$ occurs at $c_0 = 2$.
Since $F''(c) < 0$, $c_0$ is the point of maximum of $F$,
for which we have $F(c_0) = 0$. Therefore, $F(c) \leq 0$ and the only possible value of $c$
for which $D_c(k) \geq 0$ for every $k \in \mathbb{R}$ is the value $c = c_0 = 2$. Indeed,
in this case, we can factorize the dispersion relation in the form
\begin{equation}
\label{factorization-mod}
D_{c = c_0}(k) = k^2 - 2 + \frac{1}{k^2} = \frac{(1-k^2)^2}{k^2} \geq 0.
\end{equation}
The statement of the lemma is proved.
\end{proof}

The following theorem represents the main result of this paper
for the modified reduced Ostrovsky equation (\ref{redModOst}).
It shows that
the second variation of $\delta^2 \Lambda_{c,\gamma}$ is positive for every
$v \in H^1_{\rm per, zero}$ and every $\gamma > 1$ such that $|\gamma - 1|$ is sufficiently small,
if $c$ is chosen within an appropriate interval.
Since $|U'(z)| < 1$ for every $z \in \mathbb{R}$ if $\gamma \in \left(1,\frac{\pi^2}{8}\right)$,
a global solution to  the reduced
modified Ostrovsky equation (\ref{redModOst}) exists for sufficiently small
$v \in H^s_{\rm per,zero}$ with $s > \frac{3}{2}$ \cite{LPS2,PelSak}. The value of $\Lambda_{c,\gamma}$
conserves in the time evolution of the reduced modified Ostrovsky equation (\ref{redModOst}). 
Orbital stability of the periodic wave with profile $U$ with respect to subharmonic perturbations
$v \in H^s_{\rm per, zero}$ with $s > \frac{3}{2}$  follows from
the positivity of $\delta^2 \Lambda_{c,\gamma}$ \cite{GP1}.
The choice of $H^s_{\rm per,zero}$ with $s > \frac{3}{2}$ is explained by the necessity
to control $v_z \in L^{\infty}_{\rm per}$ by Sobolev's embedding.

\begin{theorem}
\label{theorem-red-Mod-Ost}
There exists $\gamma_0 \in \left(1,\frac{\pi^2}{8}\right)$ and $C > 0$ such that
$\delta^2 \Lambda_{c,\gamma} \geq C \| v \|_{H^1_{\rm per}}^2$
for every $\gamma \in (1,\gamma_0)$ and every $v \in H^1_{\rm per, zero}$
such that $\langle U', v \rangle_{L^2_{\rm per}} = 0$, if $c \in (c_-,c_+)$, where $c_{\pm}$
are given by the asymptotic expansion
\begin{equation}
\label{exact-c-plus-minus-mod}
c_{\pm} = 2 \pm 2a + \mathcal{O}(a^2),
\end{equation}
where $a$ determines $\gamma$ and $I$ in the Stokes expansions
(\ref{gamma-parameter-mod}) and (\ref{E-parameter-mod}).
\end{theorem}

\begin{remark}
Using numerical computations, we will show in Section 5 that the result of Theorem \ref{theorem-red-Mod-Ost}
extends to every $\gamma$ in the interval $\left(1,\frac{\pi^2}{8}\right)$.
\end{remark}

\section{Floquet--Bloch bands for periodic waves of small amplitudes}

We consider the second variation of the Lyapunov functional $\Lambda_{c,\gamma}$ defined by (\ref{Lyapunov})
at the periodic wave profile $U$ for nonzero but small amplitude parameter $a$. In order to prove
positivity of $\delta^2 \Lambda_{c,\gamma}$ and thus to prove Theorems \ref{theorem-red-Ost} and \ref{theorem-red-Mod-Ost},
we develop a general perturbation method, which is applied to both versions of the reduced Ostrovsky equations.

\subsection{General perturbation method}

We are interested to characterize the spectrum of the linear operator $K_{c,\gamma} := L_{\gamma} - c M_{\gamma}$ in
$L^2_{\rm per,zero}$ with the domain $X_{\rm per, zero} \subset L^2_{\rm per,zero}$, where
the linear operators $L_{\gamma}$ and $M_{\gamma}$ define the quadratic forms (\ref{operator-L})
and (\ref{operator-M}). The parameter $c \in \mathbb{R}$ is to be defined within an appropriate
interval. Here $L^2_{\rm per,zero}$ denotes the space of $2\pi N$-periodic functions with zero mean,
where $N$ is a positive integer, whereas $X_{\rm per, zero}$ is the domain of the linear operator $K_{c,\gamma}$
given by a Sobolev space $H^s_{\rm per, zero}$ for some $s \geq 0$.
In order to obtain results uniformly in $N$, we consider
the Floquet--Bloch spectrum of $K_{c,\gamma}$ in $L^2(\mathbb{R})$ with the domain $X(\R) \subset L^2(\mathbb{R})$.

By construction, $K_{c,\gamma}$ is a self-adjoint operator in $L^2(\mathbb{R})$ with
$2\pi$-periodic coefficients. By the Floquet--Bloch theory,
we look for $2\pi$-periodic Bloch wave functions $w(\cdot,\kappa) \in
X_{\rm per}$ with the quasi-momentum parameter $\kappa$ defined in the Brillouin zone
$\mathbb{T} = \left[-\frac{1}{2},\frac{1}{2}\right]$ such that
$$
e^{i \kappa z} w(z,\kappa) \in L^\infty(\R) \cap X_{\rm loc}(\R)
$$
is an eigenfunction of $K_{c,\gamma}$ for an eigenvalue $\lambda(\kappa)$. We say that
$\lambda(\kappa)$ belongs to the Floquet--Bloch spectrum of $K_{c,\gamma}$ in $L^2(\mathbb{R})$.
Both $w(z,\kappa)$ and $\lambda(\kappa)$ depend also on parameters $c$ and $\gamma$
but we neglect listing this dependence explicitly.

Let us further postulate the Stokes expansion for the periodic wave profile $U$
in the normalized form:
\begin{equation}
\label{Stokes-again}
U(z) = a \cos(z) + a^2 \tilde{U}_a(z), \quad \gamma = 1 + a^2 \tilde{\gamma}_a,
\end{equation}
where $\tilde{U}_a$ is an even $2\pi$-periodic function such that $\langle \cos(\cdot), \tilde{U}_a \rangle_{L^2_{\rm per}} = 0$
and $a^2 \tilde{U}_a = \mathcal{O}_{L^2_{\rm per}}(a^2)$ as $a \to 0$,
whereas $a^2 \tilde{\gamma}_a = \mathcal{O}(a^2)$ as $a \to 0$.

Let us denote $P_{c,\gamma}(\kappa) := e^{-i \kappa z} K_{c,\gamma} e^{i \kappa z}$.
Assuming smoothness of $P_{c,\gamma}(\kappa)$ with respect to the small amplitude parameter $a$,
we expand it in powers of $a$:
\begin{equation}
\label{P-c-expansion}
P_{c,\gamma}(\kappa) = P^{(0)}_{c}(\kappa) + a  P^{(1)}_c(\kappa) + a^2 P^{(2)}_c(\kappa) +
\mathcal{O}_{X_{\rm per} \to L^2_{\rm per}}(a^3).
\end{equation}
The operator $P^{(0)}_c(\kappa)$ has constant coefficients, and its spectrum
in the space $L^2_{\rm per}$ consists of a countable family of
real eigenvalues $\{\lambda_n^{(0)}(\kappa) \}_{n \in \mathbb{Z}}$.
The following assumption ensures that there exist two unperturbed Floquet--Bloch
bands, which touch zero at $\kappa = 0$ as convex functions, if $c = c_0$,
whereas all other bands are bounded away from zero by a positive number, see, e.g.,
the results of Lemmas \ref{lemma-positivity} and \ref{lemma-positivity-mod}.

\begin{assumption}
There exists $c_0$ such that for $c = c_0$, we have $\lambda_n^{(0)}(\kappa) \ge 0$ for all $n
\in \mathbb{Z}$ and all $\kappa \in \mathbb{T}$. Moreover, there exist exactly two bands
$\lambda_{\pm 1}^{(0)}(\kappa)$ which are smooth functions in $\kappa$ and behave like
\begin{equation}
\label{bands-convex}
\lambda_{\pm 1}^{(0)}(\kappa) = \frac{1}{2} \lambda''(0) \kappa^2 + \mathcal{O}(\kappa^3) \quad \mbox{\rm as}
\quad \kappa \to 0,
\end{equation}
where $\lambda''(0) > 0$. Furthermore, there exists $C > 0$ such that
\begin{equation}
\label{bands-bounded}
\lambda_n^{(0)}(\kappa) \geq C, \quad \mbox{\rm for all \;} n \in \mathbb{Z} \backslash \{+1,-1\}
\mbox{\rm \; and all \;} \kappa \in \mathbb{T}.
\end{equation}
\label{assumption-1}
\end{assumption}

By translational symmetry, we know that $P_{c,\gamma}(0) U' = 0$ for every $c \in \mathbb{R}$
and every $\gamma$ such that $|\gamma - 1|$ is sufficiently small,
see, e.g., the proofs of Lemmas \ref{lemma-spectrum-S} and \ref{lemma-spectrum-R}.
As a result, there exists at least one Floquet--Bloch band
of $P_{c,\gamma}(\kappa)$ that touches zero at $\kappa = 0$ for every $c$ and $\gamma$.
The other zero eigenvalue of $K_{c,\gamma} = P_{c,\gamma}(0)$ exists at $a = 0$ but is supposed to shift to a small
positive number for nonzero but small $a$, according to the following assumption.

\begin{assumption}
For every $c$ such that $|c-c_0|$ is sufficiently small,
there exists two eigenvalues of $K_{c,\gamma} = P_{c,\gamma}(0)$ in $L^2_{\rm per}$ in the neighborhood of zero.
One eigenvalue is identically zero for every small nonzero $a$, whereas the other eigenvalue
$\lambda(a,c)$ is strictly positive and satisfies the asymptotic expansion
\begin{equation}
\label{eigenvalue-positive}
\lambda(a,c) = \lambda_2(c) a^2 + a^3 \tilde{\lambda}(a,c),
\end{equation}
where $\lambda_2(c)$ is smooth in $c$ with $\lambda_2(c_0) > 0$,
whereas $\tilde{\lambda}(a,c)$ is smooth in $(a,c)$ and bounded as $a \to 0$ and $c \to c_0$.
\label{assumption-2}
\end{assumption}

Let $W_a$ be the eigenfunction of $K_{c,\gamma}$ for the zero eigenvalue,
which is independent of $c$. Let $V_{a,c}$ be the eigenfunction of $K_{c,\gamma}$
for the eigenvalue $\lambda(a,c)$ given by (\ref{eigenvalue-positive})
in Assumption \ref{assumption-2}. From the Stokes expansion (\ref{Stokes-again}),
we can set the normalized eigenfunctions to the form
\begin{equation}
\label{eigenfunction-expansions}
W_a(z) := -a^{-1} U'(z) = \sin(z) + a \tilde{W}_a(z), \quad
V_{a,c}(z) = \cos(z) + a \tilde{V}_{a,c}(z),
\end{equation}
where the correction terms $\tilde{W}_a$ and $\tilde{V}_{a,c}$ are uniquely defined
and bounded in $L^2_{\rm per}$ as $a \to 0$ and $c \to c_0$.
Finally, we assume a technical non-degeneracy condition.

\begin{assumption}
Besides $P_{c,\gamma}(0) W_a = 0$ and $P_{c,\gamma}(0) V_{a,c} = \lambda(a,c) V_{a,c}$, we have
\begin{equation}
\label{non-degeneracy}
P_{c,\gamma}'(0) W_a = -i \mu_1 (c - c_0) \cos(z) + a F_c(z) + a^2 G_{a,c}(z),
\end{equation}
and
\begin{equation}
\label{non-degeneracy-ex}
P_{c,\gamma}'(0) V_{a,c} = i \mu_1 (c - c_0) \sin(z) + a \tilde{F}_c(z) + a^2 \tilde{G}_{a,c}(z),
\end{equation}
where $\mu_1$ is a real nonzero number, $F_c$ and $\tilde{F}_c$ are respectively even and odd functions such that
$$
\langle \cos(\cdot), F_c \rangle_{L^2_{\rm per}} = 0 \quad \mbox{\rm and} \quad
\langle \sin(\cdot), F_c \rangle_{L^2_{\rm per}} = 0,
$$
both functions are bounded in $L^2_{\rm per}$ as $c \to c_0$, whereas $G_{a,c}$ and $\tilde{G}_{a,c}$ are
bounded functions in $L^2_{\rm per}$ as $a \to 0$ and $c \to c_0$.
\label{assumption-3}
\end{assumption}

Under Assumptions \ref{assumption-1}, \ref{assumption-2}, and \ref{assumption-3}, we obtain
a sufficient condition that the two bands in Assumption \ref{assumption-1}
separate from each other for nonzero but small values of $a$. If $c$ is selected near $c_0$, 
one band still touches the origin at $\kappa = 0$ while the other one remains
strictly positive for all $\kappa \in \mathbb{T}$, with both bands remaining convex functions of $\kappa$.
The following proposition gives the general result.

\begin{proposition}
\label{proposition-main}
Consider expansion (\ref{P-c-expansion}) of the self-adjoint operator $P_{c,\gamma}(\kappa)$
for sufficiently small amplitude $a$ and assume that Assumptions
\ref{assumption-1}, \ref{assumption-2}, and \ref{assumption-3} are true.
Then, for every $c$ satisfying $|c - c_0| \leq C |a|$, where $C$ is a positive $a$-independent constant,
the two lowest spectral bands of the operator $P_{c,\gamma}(\kappa)$ denoted
by $\lambda_{\rm gr}(\kappa)$ and $\lambda_{\rm ex}(\kappa)$
satisfy the asymptotic expansions
\begin{equation}
\label{band-perturbation-lowest}
\lambda_{\rm gr}(\kappa) = \left( \frac{1}{2} \lambda''(0) - \frac{\mu_1^2 (c-c_0)^2}{\lambda_2(c_0) a^2} + \mathcal{O}(a) \right) \kappa^2
+ \kappa^3 \tilde{\lambda}_{\rm gr}
\end{equation}
and
\begin{equation}
\label{band-perturbation-second}
\lambda_{\rm ex}(\kappa) = \lambda_2(c_0) a^2 + \mathcal{O}(a^3) + \left( \frac{1}{2} \lambda''(0) +
\frac{\mu_1^2 (c-c_0)^2}{\lambda_2(c_0) a^2}  + \mathcal{O}(a) \right) \kappa^2
+ \kappa^3 \tilde{\lambda}_{\rm ex},
\end{equation}
where $\tilde{\lambda}_{\rm gr}$ and $\tilde{\lambda}_{\rm ex}$ are bounded as $a \to 0$ and $\kappa \to 0$.
\end{proposition}

\begin{proof}
From (\ref{bands-bounded}) we know that, if $|c - c_0|$ is sufficiently
small, there exists a constant $C > 0$ (independent of $c$) such that
\begin{equation}
\label{resolvent}
0 < \left[ \lambda_n^{(0)}(\kappa) \right]^{-1} \leq C \quad \hbox{for all } n \in \mathbb{Z}
\setminus\{+1,-1\} \hbox{ and all } \kappa \in \mathbb{T}.
\end{equation}
By the regular perturbation theory, this bound remains true (with
possibly a larger constant $C$) for all the perturbed spectral bands
with $n \neq \pm 1$ if
$a$ and $|c-c_0|$ are small enough. For the other two spectral bands in (\ref{bands-convex}),
the same perturbation argument shows that they are
bounded away from zero if $|\kappa| \ge \kappa_0$ for $a$ and $|c-c_0|$
being sufficiently small, where $\kappa_0 > 0$ is fixed independently of $a$. 
It remains to study the two spectral bands near $\kappa = 0$.

To simplify details, we consider each spectral band $\lambda_{\pm}$ separately from each other.
At the first glance, this approach does not look justified because the kernel of $P_{c=c_0,\gamma = 1}(0)$
is two-dimensional. However, due to Assumption \ref{assumption-2},
the double zero eigenvalue is broken into two simple eigenvalues of $P_{c=c_0,\gamma}$
for every $\gamma$ such that $\gamma - 1 = \mathcal{O}(a^2)$ is sufficiently small. Therefore, we can
proceed with the perturbation expansions in $\kappa$, which are singular as $a \to 0$ 
if $\kappa \neq 0$ is fixed independently of $a$. However, if $\kappa$ is as small as $\mathcal{O}(a^2)$, 
all terms of the asymptotic expansions (\ref{band-perturbation-lowest})
and (\ref{band-perturbation-second}) are controlled in the limit $a \to 0$.

First, we derive the expansion (\ref{band-perturbation-lowest}) formally.
Since $P_{c,\gamma}(0) W_a = 0$, we consider the following expansion
for the band $\lambda_{\rm gr}(\kappa)$ that touches the zero eigenvalue:
\begin{equation}
\label{band-expansion}
\lambda_{\rm gr}(\kappa) = \Lambda_1 \kappa + \Lambda_2 \kappa^2 + \mathcal{O}(\kappa^3), \quad
w(z,\kappa) = W_a(z) + \kappa W_1(z) + \kappa^2 W_2(z) + \mathcal{O}_{L^2_{\rm per}}(\kappa^3),
\end{equation}
where corrections $W_{1,2}$ are to be determined by a projection algorithm subject to the orthogonality conditions
$\langle W_a, W_{1,2} \rangle_{L^2_{\rm per}} = 0$.
At the first order in $\kappa$, we obtain the linear inhomogeneous equation
\begin{eqnarray}
P_{c,\gamma}(0) W_1 + P_{c,\gamma}'(0) W_a = \Lambda_1 W_a.
\label{perturbed-matrix-first}
\end{eqnarray}
We note that $P_{c,\gamma}'(0) = i Q_{c,\gamma}$, where $Q_{c,\gamma} := K_{c,\gamma} z - z K_{c,\gamma}$
is skew-adjoint in $L^2_{\rm per}$. Since $W_a$ and $Q_{c,\gamma}$ are real, we have
$$
\Lambda_1 \| W_a \|_{L^2_{\rm per}}^2 = i \langle Q_{c,\gamma} W_a, W_a \rangle_{L^2_{\rm per}}
= - i \langle W_a, Q_{c,\gamma} W_a \rangle_{L^2_{\rm per}} = 0,
$$
so that $\Lambda_1 = 0$.

From the asymptotic expansions (\ref{eigenvalue-positive}) in Assumption \ref{assumption-2}
and (\ref{non-degeneracy}) in Assumption \ref{assumption-3}, we obtain the unique solution
of the inhomogeneous equation (\ref{perturbed-matrix-first}) with $\Lambda_1 = 0$ in the form
\begin{eqnarray}
\nonumber
W_1 & := & -[P_{c,\gamma}(0)]^{-1} P_{c,\gamma}'(0) W_a \\
& = & \frac{i \mu_1 (c-c_0)}{\lambda_2(c) a^2} \left[ 1 + \mathcal{O}(a) \right] \cos(z) + \tilde{W}^{(1)}_{a,c},
\label{first-order-correction}
\end{eqnarray}
where $\tilde{W}^{(1)}_{a,c}$ is bounded in $L^2_{\rm per}$ as $a \to 0$ and $c \to c_0$.
If $|c-c_0| \leq C |a|$ and $|\kappa| \leq C a^2$ for an $a$-independent positive constant $C$,
then $\kappa W_1 = \mathcal{O}_{L^2_{\rm per}}(a)$ is small.

At the second order in $\kappa$, we obtain the linear inhomogeneous equation
\begin{eqnarray}
P_{c,\gamma}(0) W_2 + P_{c,\gamma}'(0) W_1 + \frac{1}{2} P_{c,\gamma}''(0) W_a = \Lambda_2 W_a.
\label{perturbed-matrix-second}
\end{eqnarray}
Projection to $W_a$ now yields
\begin{equation}
\label{zero-part}
\Lambda_2 \| W_a \|_{L^2_{\rm per}}^2 = \langle P_{c,\gamma}'(0) W_1, W_a \rangle_{L^2_{\rm per}} +
\frac{1}{2} \langle P_{c,\gamma}''(0) W_a, W_a \rangle_{L^2_{\rm per}}.
\end{equation}
Assumption \ref{assumption-1} implies that
\begin{equation}
\label{first-part}
\frac{1}{2} \langle P_{c,\gamma}''(0) W_a, W_a \rangle_{L^2_{\rm per}} = \frac{1}{2} \lambda''(0) \| W_a \|_{L^2_{\rm per}}^2
\left[ 1 + \mathcal{O}(a,c-c_0) \right].
\end{equation}
On the other hand, the explicit expressions (\ref{non-degeneracy}) and (\ref{first-order-correction}) imply that
\begin{eqnarray}
\nonumber
\langle P_{c,\gamma}'(0) W_1, W_a \rangle_{L^2_{\rm per}} & = &
\langle W_1, P_{c,\gamma}'(0) W_a \rangle_{L^2_{\rm per}} \\
\label{second-part}
& = & -\frac{\mu_1^2 (c-c_0)^2}{\lambda_2(c_0) a^2} \| V_{a,c} \|_{L^2_{\rm per}}^2
\left[ 1 + \mathcal{O}(a) \right]
+ \mathcal{O}(a,c-c_0),
\end{eqnarray}
where the leading term of the expansion is bounded as $a \to 0$ if $|c-c_0| \leq C |a|$
for an $a$-independent positive constant $C$.

From (\ref{eigenfunction-expansions}), we have
 $\| W_a \|_{L^2_{\rm per}}^2 = \pi + \mathcal{O}(a)$ and $\| V_{a,c} \|_{L^2_{\rm per}}^2 = \pi + \mathcal{O}(a)$.
Combining (\ref{band-expansion}), (\ref{zero-part}), (\ref{first-part}), and (\ref{second-part}),
we obtain (\ref{band-perturbation-lowest}).

Next, we derive the expansion (\ref{band-perturbation-second}) formally.
Since $P_{c,\gamma}(0) V_{a,c} = \lambda(a,c) V_{a,c}$, we consider the following expansion
for the band $\lambda_{\rm ex}(\kappa)$:
\begin{equation}
\label{band-expansion-ex}
\lambda_{\rm ex}(\kappa) = \lambda(a,c) + \Lambda_1 \kappa + \Lambda_2 \kappa^2 + \mathcal{O}(\kappa^3), \quad
w(z,\kappa) = V_{a,c}(z) + \kappa V_1(z) + \kappa^2 V_2(z) + \mathcal{O}_{L^2_{\rm per}}(\kappa^3),
\end{equation}
where corrections $V_{1,2}$ are to be determined by a projection algorithm subject to the orthogonality conditions
$\langle V_{a,c}, V_{1,2} \rangle_{L^2_{\rm per}} = 0$.
At the first order in $\kappa$, we obtain the linear inhomogeneous equation
\begin{eqnarray}
\left[ P_{c,\gamma}(0) - \lambda(a,c) \right] V_1 + P_{c,\gamma}'(0) V_{a,c} = \Lambda_1 V_{a,c}.
\label{perturbed-matrix-first-ex}
\end{eqnarray}
Recalling that $P_{c,\gamma}'(0) = i Q_{c, \gamma}$, projection to $V_{a,c}$ yields
$$
\Lambda_1 \| V_{a,c} \|_{L^2_{\rm per}}^2 = i \langle Q_{c,\gamma} V_{a,c}, V_{a,c} \rangle_{L^2_{\rm per}}
= - i \langle V_{a,c}, Q_{c,\gamma} V_{a,c} \rangle_{L^2_{\rm per}} = 0,
$$
so that $\Lambda_1 = 0$. From the asymptotic expansions (\ref{eigenvalue-positive}) in Assumption \ref{assumption-2}
and (\ref{non-degeneracy-ex}) in Assumption \ref{assumption-3}, we obtain the solution of
the inhomogeneous equation (\ref{perturbed-matrix-first-ex})
with $\Lambda_1 = 0$ in the form
\begin{eqnarray}
\nonumber
V_1 & := & -[P_{c,\gamma}(0) - \lambda(a,c)]^{-1} P_{c,\gamma}'(0) V_{a,c} \\
& = & \frac{i \mu_1 (c-c_0)}{\lambda_2(c) a^2} \left[ 1 + \mathcal{O}(a) \right] \sin(z) + \tilde{V}^{(1)}_{a,c},
\label{first-order-correction-ex}
\end{eqnarray}
where $\tilde{V}^{(1)}_{a,c}$ is bounded in $L^2_{\rm per}$ as $a \to 0$ and $c \to c_0$. Again,
the correction term $\kappa V_1 = \mathcal{O}_{L^2_{\rm per}}(a)$ is small, if $|c-c_0| \leq C |a|$ 
and $|\kappa| \leq C a^2$ for an $a$-independent positive constant $C$.

At the second order in $\kappa$, we obtain the linear inhomogeneous equation
\begin{eqnarray}
\left[ P_{c,\gamma}(0) - \lambda(a,c) \right] V_2 +
P_{c,\gamma}'(0) V_1 + \frac{1}{2} P_{c,\gamma}''(0) V_{a,c} = \Lambda_2 V_{a,c}.
\label{perturbed-matrix-second-ex}
\end{eqnarray}
Projection to $V_{a,c}$ yields
\begin{equation}
\label{zero-part-ex}
\Lambda_2 \| V_{a,c} \|_{L^2_{\rm per}}^2 = \langle P_{c,\gamma}'(0) V_1, V_{a,c} \rangle_{L^2_{\rm per}} +
\frac{1}{2} \langle P_{c,\gamma}''(0) V_{a,c}, V_{a,c} \rangle_{L^2_{\rm per}}.
\end{equation}
Assumption \ref{assumption-1} implies that
\begin{equation}
\label{first-part-ex}
\frac{1}{2} \langle P_{c,\gamma}''(0) V_{a,c}, V_{a,c} \rangle_{L^2_{\rm per}} = \frac{1}{2} \lambda''(0) \| V_{a,c} \|_{L^2_{\rm per}}^2
\left[ 1 + \mathcal{O}(a,c-c_0) \right].
\end{equation}
On the other hand, the explicit expressions (\ref{non-degeneracy-ex}) and (\ref{first-order-correction-ex}) imply that
\begin{eqnarray}
\nonumber
\langle P_{c,\gamma}'(0) V_1, V_{a,c} \rangle_{L^2_{\rm per}} & = &
\langle V_1, P_{c,\gamma}'(0) V_{a,c} \rangle_{L^2_{\rm per}} \\
\label{second-part-ex}
& = & \frac{\mu_1^2 (c-c_0)^2}{\lambda_2(c_0) a^2} \| W_a \|_{L^2_{\rm per}}^2
\left[ 1 + \mathcal{O}(a) \right]
+ \mathcal{O}(a,c-c_0),
\end{eqnarray}
where the leading term of the expansion is bounded as $a \to 0$ if $|c-c_0| \leq C |a|$
for an $a$-independent positive constant $C$.
Combining (\ref{band-expansion-ex}), (\ref{zero-part-ex}), (\ref{first-part-ex}), and (\ref{second-part-ex}),
we obtain (\ref{band-perturbation-second}).

It remains to justify the expansions (\ref{band-perturbation-lowest}) and (\ref{band-perturbation-second}).
Since the method is similar, we only report justification of the expansion (\ref{band-perturbation-lowest}).
Using scaling $\lambda_{\rm gr}(\kappa) = \kappa^2 \Lambda$, we substitute the orthogonal decomposition
\begin{equation}
\label{orth0}
w = W_a(z) + b V_{a,c}(z) + \tilde{w}(z), \quad \langle W_a,\tilde{w} \rangle_{L^2_{\rm per}} = \langle V_{a,c}, \tilde{w} \rangle_{L^2_{\rm per}} = 0,
\end{equation}
to the spectral problem $P_{c,\gamma}(\kappa) w = \kappa^2 \Lambda w$. By projecting the spectral
problem to $W_a$ and $V_{a,c}$ and using the previous relations, we obtain two equations
\begin{eqnarray}
\nonumber
\kappa^2 \Lambda \| W_a \|_{L^2_{\rm per}}^2 & = & \langle \left( \frac{1}{2} \kappa^2 P_{c,\gamma}''(0) + \mathcal{O}(\kappa^3) \right) W_a, W_a \rangle_{L^2_{\rm per}} + b \langle \left( \kappa P_{c,\gamma}'(0) + \mathcal{O}(\kappa^2) \right) V_{a,c}, W_a \rangle_{L^2_{\rm per}} \\
\label{orth1} & \phantom{t} &
+ \langle \left( \kappa P_{c,\gamma}'(0) + \mathcal{O}(\kappa^2) \right) \tilde{w}, W_a \rangle_{L^2_{\rm per}}
\end{eqnarray}
and
\begin{eqnarray}
\nonumber
\kappa^2 \Lambda b \| V_{a,c} \|_{L^2_{\rm per}}^2 & = & \langle \left( \kappa P_{c,\gamma}'(0) + \mathcal{O}(\kappa^2) \right) W_a, V_{a,c} \rangle_{L^2_{\rm per}} + b \langle \left( P_{c,\gamma}(0) + \mathcal{O}(\kappa) \right) V_{a,c}, V_{a,c} \rangle_{L^2_{\rm per}} \\
\label{orth2} & \phantom{t} &
+ \langle \left( P_{c,\gamma}(0) + \mathcal{O}(\kappa) \right) \tilde{w}, V_{a,c} \rangle_{L^2_{\rm per}},
\end{eqnarray}
where all correction terms to the linear operators from $X_{\rm per}$ to $L^2_{\rm per}$ are defined in the operator norm.
The residual problem for $\tilde{w}$ is written in the form
\begin{eqnarray}
\label{orth3}
\left( P_{c,\gamma}(\kappa) - \kappa^2 \Lambda \right) \tilde{w} & = & \kappa^2 \Lambda (W_a + b V_{a,c})
- P_{c,\gamma}(\kappa) ( W_a + b V_{a,c} ).
\end{eqnarray}

Let us now assume that $c$ and $\kappa$ satisfy $|c - c_0| \leq C |a|$ and
$|\kappa| \leq C a^2$, where $C$ is a positive $a$-independent constant.
Thanks to the bound (\ref{resolvent}), the orthogonal decomposition (\ref{orth0}), and the projection equations
(\ref{orth1}) and (\ref{orth2}), for every $\Lambda = \mathcal{O}(1)$ and $b = \mathcal{O}(1)$
as $a \to 0$, we have a unique solution to the linear inhomogeneous equation (\ref{orth3})
for $\tilde{w}$ satisfying the bound
\begin{equation}
\label{orth4}
\| \tilde{w} \|_{L^2_{\rm per}} \leq C \left( (|\kappa| |c-c_0|  + |\kappa| |a| + \kappa^2) (1 + |b|) + |b| a^2 \right),
\end{equation}
where the positive constant $C$ is $a$-independent. With the account of
Assumptions \ref{assumption-1}, \ref{assumption-2}, \ref{assumption-3}, as well as the bound (\ref{orth4}),
we obtain a unique expression for $b$ from the projection equation (\ref{orth2})
for every $|\Lambda| = \mathcal{O}(1)$ as $a \to 0$:
\begin{eqnarray}
\label{orth5}
b \left( \lambda_2(c_0) a^2 + \mathcal{O}(a^3,a^2|c-c_0|,\kappa^2) \right) \| V_{a,c} \|_{L^2_{\rm per}}^2 =
i \mu_1 (c-c_0) \kappa \| V_{a,c} \|_{L^2_{\rm per}}^2 + \mathcal{O}(a^2 \kappa,\kappa^2).
\end{eqnarray}
In view of the constraints on $c$ and $\kappa$, this yields
\begin{eqnarray}
\label{orth6}
b = \frac{i \mu_1 (c-c_0) \kappa}{\lambda_2(c_0) a^2} \left(1 + \mathcal{O}(a) \right),
\end{eqnarray}
so that $b = \mathcal{O}(a)$  as $a \to 0$. Finally, substituting $\tilde{w}$ and $b$ satisfying
(\ref{orth4}) and (\ref{orth6}) to the projection equation (\ref{orth1}) yields the unique
expressions for $\Lambda$:
\begin{eqnarray}
\label{orth7}
\Lambda = \frac{1}{2} \lambda''(0) - \frac{\mu_1^2 (c-c_0)^2}{\lambda_2(c_0) a^2} + \mathcal{O}(a,\kappa),
\end{eqnarray}
from which the expansion (\ref{band-perturbation-lowest}) follows. In particular, we see that 
$$
\lambda_{\rm gr}''(0) = \lambda''(0) - \frac{2 \mu_1^2 (c-c_0)^2}{\lambda_2(c_0) a^2} + \mathcal{O}(a),
$$
where the leading term is $\mathcal{O}(1)$ as $a \to 0$. From smoothness of $\lambda_{\rm gr}$ in $\kappa$ 
for every $a \neq 0$ sufficiently small, the expansion (\ref{band-perturbation-lowest}) is justified 
for every $\kappa \neq 0$ sufficiently small. Justification of the expansion (\ref{band-perturbation-second}) for 
$\lambda_{\rm ex}$ is analogous. 
\end{proof}

\begin{corollary}
It follows from the asymptotic expansions (\ref{band-perturbation-lowest}) and (\ref{band-perturbation-second})
of Proposition \ref{proposition-main} that the self-adjoint operator $K_{c,\gamma}$ is positive 
if $c \in (c_-,c_+)$, where $c_{\pm}$ satisfy the asymptotic expansion
\begin{equation}
c_{\pm} = c_0 + \frac{\sqrt{\lambda_2(c_0) \lambda''(0)}}{\sqrt{2} \mu_1} a + \mathcal{O}(a^2) \quad \mbox{\rm as} \quad a \to 0.
\end{equation}
\end{corollary}

\begin{remark}
Note that the proof of Proposition \ref{proposition-main} does not rely on 
the Lyapunov--Schmidt method applied to the two-dimensional kernel of $P_c^{(0)}(0)$
spanned by $\{ \cos(\cdot),\sin(\cdot) \} \subset L^2_{\rm per}$, compared to the one used in the recent work \cite{GP1}.
Although the same two-dimensional Lyapunov--Schmidt method applies relatively easy for the case of cubic nonlinearities,
where the terms of the order $\mathcal{O}(a^2)$ are missing in the Stokes expansion (\ref{Stokes-mod}),
the method becomes messy and requires an additional near-identity transformation for
the case of quadratic nonlinearities, which is associated with the Stokes expansion (\ref{Stokes}).
\end{remark}

\subsection{Proof of Theorem \ref{theorem-red-Ost}}

Here we report computations, which verify Assumptions \ref{assumption-1}, \ref{assumption-2}, and \ref{assumption-3}
for the case of the reduced Ostrovsky equation (\ref{redOst}). Let $U$ be the $2\pi$-periodic
smooth solutions of the differential equation (\ref{second-order}) given by Lemma \ref{lemma-small-amplitude}.
The Stokes expansion (\ref{Stokes-again}) is satisfied with
$$
\tilde{U}_a(z) = \frac{1}{3} \cos(2z) + \frac{3}{16} a \cos(3z) + \mathcal{O}_{C^{\infty}_{\rm per}}(a^2),
\quad \tilde{\gamma}_a = \frac{1}{6} + \mathcal{O}(a^2).
$$
The second variation of the Lyapunov functional $\Lambda_{c,\gamma}$ defined by (\ref{Lyapunov}) is given by
\begin{equation}
\label{Lyapunov-second-variation}
\delta^2 \Lambda_{c,\gamma} = \int \left[ (\partial_z^{-1} v)^2 - (\gamma - U) v^2 - \frac{c}{(\gamma^3 - 6I)^{2/3}} v^2 +
\frac{c (\gamma - U)^5}{(\gamma^3 - 6I)^{5/3}} (\partial_z^2 v)^2 \right] dz,
\end{equation}
where we have used (\ref{identity}), (\ref{delta-2-S}), and (\ref{delta-2-R}).
The second variation is generated by the self-adjoint operator $K_{c,\gamma}$ with the domain
$X_{\rm per, zero} = H^4_{\rm per,zero} \subset L^2_{\rm per,zero}$,
where $K_{c,\gamma}$ is given explicitly by
\begin{equation}
K_{c,\gamma} := -\partial_z^{-2} - (\gamma - U) - \frac{c}{(\gamma^3 - 6I)^{2/3}} +
\frac{c}{(\gamma^3 - 6I)^{5/3}} \partial_z^2 (\gamma - U)^5 \partial_z^2.
\end{equation}
From here, we define $P_{c,\gamma}(\kappa) = e^{-i \kappa z} K_{c,\gamma} e^{i \kappa z}$, or explicitly,
\begin{equation}
\label{unperturbed-operator-perturbed}
P_{c,\gamma}(\kappa) = -(\partial_z + i \kappa)^{-2}
- (\gamma - U) - \frac{c}{(\gamma^3 - 6I)^{2/3}} +
\frac{c}{(\gamma^3 - 6I)^{5/3}} (\partial_z+i\kappa)^2 (\gamma - U)^5 (\partial_z+i \kappa)^2.
\end{equation}
By using expansion (\ref{P-c-expansion}), we obtain the unperturbed operator
\begin{equation}
\label{unperturbed-operator}
P_c^{(0)}(\kappa) := -(\partial_z + i \kappa)^{-2} - 1 - c + c (\partial_z+i\kappa)^4,
\end{equation}
which corresponds to the quadratic form (\ref{Lyapunov-zero}) modified with the Floquet--Bloch parameter
$\kappa \in \mathbb{T}$.
Assumption \ref{assumption-1} is verified by Lemma \ref{lemma-positivity} with $c_0 = \frac{1}{2}$.
We obtain from (\ref{factorization}) with $\kappa := k \mp 1$,
$$
\lambda_{\pm 1}^{(0)}(\kappa) = \frac{(2 \pm \kappa)^2 (3 \pm 2 \kappa + \kappa^2)}{2 (1 \pm \kappa)^2} \kappa^2,
$$
from which we have $\lambda''(0) = 12$.

Assumption \ref{assumption-2} is verified by Lemmas \ref{lemma-spectrum-S} and \ref{lemma-spectrum-R}.
Since the correction term $V_1$ is the same in the asymptotic expansions (\ref{asympt-V-1}) and (\ref{asympt-V-1-R}),
see expressions (\ref{lambda-1-solution}) and (\ref{lambda-1-solution-M}), we obtain $\lambda_2(c)$
by the linear superposition of (\ref{lambda-2-S}) and (\ref{lambda-2-R}):
$$
\lambda_2(c) = \frac{1+10c}{3}.
$$
In particular, we have $\lambda_2(c_0) = 2$. The expansions (\ref{eigenfunction-expansions}) hold with
$$
\tilde{W}_a(z) = \frac{2}{3} \sin(2z) + \mathcal{O}_{L^2_{\rm per}}(a), \quad
\tilde{V}_{a,c}(z) = \frac{2}{3} \cos(2z) + \mathcal{O}_{L^2_{\rm per}}(a).
$$

In order to verify Assumption \ref{assumption-3}, we differentiate (\ref{unperturbed-operator-perturbed}) in $\kappa$ and
obtain
\begin{equation}
\label{derivative-unperturbed-operator}
P_{c,\gamma}'(0) = 2i \partial_z^{-3} + \frac{2 i c}{(\gamma^3-6I)^{5/3}} \left[ \partial_z (\gamma - U)^5 \partial_z^2
+ \partial_z^2 (\gamma - U)^5 \partial_z \right].
\end{equation}
Applying $P_{c,\gamma = 1}'(0)$ to $W_a$ and $V_{a,c}$ given by the expansions (\ref{eigenfunction-expansions}),
we obtain expansions (\ref{non-degeneracy}) and (\ref{non-degeneracy-ex}) with $\mu_1 = 4$ and
$$
F_c(z) = \frac{i}{6} (1 + 52c) \cos(2z), \quad \tilde{F}_c(z) = -\frac{i}{6} (1 + 52c) \sin(2z).
$$

Combining all together, we apply the result of Proposition \ref{proposition-main} and
obtain the expansion (\ref{band-perturbation-lowest}) for the spectral band $\lambda_{\rm gr}$ in the explicit form
\begin{equation}
\label{band-perturbation-lowest-red-Ost}
\lambda_{\rm gr}(\kappa) = \left( 6 - \frac{2 (2c-1)^2}{a^2} + \mathcal{O}(a) \right) \kappa^2
+ \mathcal{O}(\kappa^3) \quad \mbox{\rm as} \quad \kappa \to 0, \quad a \to 0.
\end{equation}
Therefore, the positivity is proved for $(2c-1)^2 < 3 a^2 + \mathcal{O}(a^3)$ or
$|2c-1| < \sqrt{3} |a| + \mathcal{O}(a^3)$, which proves the expansion (\ref{exact-c-plus-minus})
in Theorem \ref{theorem-red-Ost}. The other spectral band $\lambda_{\rm ex}$ with the expansion
(\ref{band-perturbation-second}) is strictly positive if $a \neq 0$ is sufficiently small.

\subsection{Proof of Theorem \ref{theorem-red-Mod-Ost}}

Here we report computations, which verify Assumptions \ref{assumption-1}, \ref{assumption-2}, and \ref{assumption-3}
for the case of the reduced modified Ostrovsky equation (\ref{redModOst}).
Let $U$ be the $2\pi$-periodic smooth solution of the differential equation (\ref{second-order-mod}) given by
Lemma \ref{lemma-small-amplitude-mod}. The Stokes expansion (\ref{Stokes-again}) is satisfied with
$$
\tilde{U}_a(z) = \frac{3}{64} a \cos(3z) + \mathcal{O}_{C^{\infty}_{\rm per}}(a^3), \quad
\tilde{\gamma}_a = \frac{1}{8} + \mathcal{O}(a^2).
$$
The second variation of the Lyapunov functional $\Lambda_{c,\gamma}$ defined by (\ref{Lyapunov}) is given by
\begin{equation}
\label{Lyapunov-second-variation-mod}
\delta^2 \Lambda_{c,\gamma} = \int \left[ (\partial_z^{-1} v)^2 - \left( \gamma - \frac{1}{2} U^2 \right) v^2
- \frac{c}{2(\gamma^2 - 2I)^{1/2}} v^2 +
\frac{c}{2 (1 - (U')^2)^{3/2}} (\partial_z v)^2 \right] dz,
\end{equation}
where we have used (\ref{delta-2-S-mod}) and (\ref{delta-2-R-mod}).
The second variation is generated by the self-adjoint operator $K_{c,\gamma}$ with the domain
$X_{\rm per. zero} = H^2_{\rm per,zero}\subset L^2_{\rm per,zero}$,
where $K_{c,\gamma}$ is given explicitly by
\begin{equation}
K_{c,\gamma} := -\partial_z^{-2} - \left(\gamma - \frac{1}{2}U^2 \right) - \frac{c}{2 (\gamma^2 - 2I)^{1/2}} -
\frac{c}{2} \partial_z (1 - (U')^2)^{-3/2} \partial_z.
\label{K_mod}
\end{equation}
From here, we define $P_{c,\gamma}(\kappa) = e^{-i \kappa z} K_{c,\gamma} e^{i \kappa z}$, or explicitly,
\begin{equation}
\label{unperturbed-operator-perturbed-modified}
P_{c,\gamma}(\kappa) = -(\partial_z + i \kappa)^{-2} - \left(\gamma - \frac{1}{2}U^2 \right) - \frac{c}{2 (\gamma^2 - 2I)^{1/2}} -
\frac{c}{2} (\partial_z + i \kappa) (1 - (U')^2)^{-3/2} (\partial_z + i \kappa).
\end{equation}
By using expansion (\ref{P-c-expansion}), we obtain the unperturbed operator
\begin{equation}
\label{unperturbed-operator-mod}
P_c^{(0)}(\kappa) := -(\partial_z + i \kappa)^{-2} - 1 - \frac{c}{2} - \frac{c}{2} (\partial_z+i\kappa)^2,
\end{equation}
which corresponds to the quadratic form (\ref{Lyapunov-zero-mod}) modified with the Floquet--Bloch parameter
$\kappa \in \mathbb{T}$.
Assumption \ref{assumption-1} is verified by Lemma \ref{lemma-positivity} with $c_0 = 2$.
We obtain from (\ref{factorization-mod}) with $\kappa := k \mp 1$,
$$
\lambda_{\pm 1}^{(0)}(\kappa) = \frac{(2 \pm \kappa)^2}{(1 \pm \kappa)^2} \kappa^2,
$$
from which we have $\lambda''(0) = 8$.

Assumption \ref{assumption-2} is verified by Lemmas \ref{lemma-spectrum-S-mod} and \ref{lemma-spectrum-R-mod}.
With the linear superposition of (\ref{lambda-2-S-mod}) and (\ref{lambda-2-R-mod}), we obtain
$$
\lambda_2(c) = \frac{2+3c}{8}.
$$
In particular, we have $\lambda_2(c_0) = 1$. The expansions (\ref{eigenfunction-expansions})
hold with $\tilde{W}_a = \mathcal{O}_{L^2_{\rm per}}(a)$, and $\tilde{V}_{a,c} = \mathcal{O}_{L^2_{\rm per}}(a)$ as $a \to 0$.

Assumption \ref{assumption-3} is verified by using the derivative operator
\begin{equation}
\label{derivative-unperturbed-operator-mod}
P_{c,\gamma}'(0) = 2i \partial_z^{-3} - \frac{i c}{2} \left[ (1 - (U')^2)^{-3/2} \partial_z
+ \partial_z (1 - (U')^2)^{-3/2} \right].
\end{equation}
Applying $P_{c,\gamma}'(0)$ to $W_a$ and $V_{a,c}$ given by the expansions (\ref{eigenfunction-expansions}),
we obtain expansions (\ref{non-degeneracy}) and (\ref{non-degeneracy-ex})
with $\mu_1 = 1$, $F_c \equiv 0$,  and $\tilde{F}_c \equiv 0$.

Combining all together, we apply the result of Proposition \ref{proposition-main} and
obtain the expansion (\ref{band-perturbation-lowest}) for the spectral band $\lambda_{\rm gr}$
in the explicit form
\begin{equation}
\label{band-perturbation-lowest-red-Ost-mod}
\lambda_{\rm gr}(\kappa) = \left( 4 - \frac{(c-2)^2}{a^2} + \mathcal{O}(a) \right) \kappa^2
+ \mathcal{O}(\kappa^3) \quad \mbox{\rm as} \quad \kappa \to 0, \quad a \to 0.
\end{equation}
Therefore, the positivity is proved for $(c-2)^2 < 4 a^2 + \mathcal{O}(a^3)$ or
$|c-2| < 2 |a| + \mathcal{O}(a^3)$, which proves the expansion (\ref{exact-c-plus-minus-mod})
in Theorem \ref{theorem-red-Mod-Ost}. The other spectral band $\lambda_{\rm ex}$ with
the expansion (\ref{band-perturbation-second}) is strictly positive if $a \neq 0$ is sufficiently small. 

\section{Numerical results for periodic waves of large amplitudes}

Here we approximate the periodic waves and the spectral bands of the linear operators $K_{\gamma,c}$
in $L^2(\mathbb{R})$ numerically for both versions of the reduced Ostrovsky equations.

Although explicit solutions to the second-order equation (\ref{second-order})
are known in the parametric form (\ref{wave-transformation}), it is more
convenient to construct the solution profile $U$ numerically. We use the Fourier series
and follow Newton-Kantorovich iterations \cite{Boyd01}. Denote the $m$-th iterate for
the solution ($U$, $\gamma$) by ($U_m$, $\gamma_m$) and introduce the increments
($v$, $g$), through $U_{m+1}=U_m+v$, $\gamma_{m+1}=\gamma_m+g$. Then,
at the linearized approximation, $(v,g)$ satisfy the linear equation
\begin{equation}
 (\gamma_m-U_m)v''-2U_m'v'+(1-U_m'')v +U_m''g = -[(\gamma_m-U_m)U_m''-(U_m')^2 + U_m].
\end{equation}
This equation is solved pseudo-spectrally. The even solution $U$ is approximated
by the $N$-th partial sum of the Fourier cosine series:
\begin{equation}
\label{partial-sum}
U(z) = \sum_{n = 1}^N A_n \cos(nz).
\end{equation}
We normalize $A_1 = a$ for a given $a \in \mathbb{R}$ and select $g$ from the orthogonality
condition 
\begin{equation}
\label{condition-on-gamma}
\langle \cos(\cdot), v \rangle_{L^2_{\rm per}} = 0. 
\end{equation}
Then, the value $A_1 = a$
is preserved in iterations in $m$. The partial sum (\ref{partial-sum})
is evaluated at $2 N$ evenly spaced points on $[-\pi,\pi)$ to give the vectors $\bmz$ and $\bmU_m$ and the matrix equation
\begin{multline}
\left(
\begin{array}{cc}
{\rm diag}(\gamma_m-\bmU_m)\,D^2 -2 (D\,\bmU_m)\,D + (\II-D^2\bmU_m) & D^2\bmU_m \\
\cos\,\bmz & 0
\end{array}
\right)
\binom{\bmv}{g} \\
= \binom{{\rm diag}(\gamma_m-\bmU_m)D^2 \bmU_m-(D\,\bmU_m)^2 + \bmU_m} {0}, \label{numerical-system}
\end{multline}
where $D$ is the $2 N \times 2 N$ Fourier differentiation matrix, $\II$ the $2 N \times 2 N$ identity matrix,
{\rm diag} denotes the $2 N \times 2 N$ matrix with the specified diagonal, and $\bmv$ gives the increment
at the points $\bmz$. The final row represents the orthogonality condition (\ref{condition-on-gamma}) numerically.

The linear algebraic system (\ref{numerical-system}) is solved iteratively with $N =2048$
until the maximum absolute value of each increment is less than $10^{-15}$,
which always required fewer than $10$ iterations. Aliasing in the computation
of products was avoided by ensuring that the coefficients satisfied $|A_n|<10^{-14}$
for $n>N/4$. The invariant \eqref{first-order} was found at each $z$ to deviate
from its average value by less than $10^{-11}$, providing an independent check on
the accuracy of the solutions. Figure \ref{f:RO} gives typical profiles $U(z)$ and the
corresponding coefficients $A_n$. Even for the largest amplitude of $a=-0.65$
considered here the coefficients $A_n$ are less than $10^{-16}$ for $n>300$.

  \begin{figure}
    \includegraphics[width=0.46\textwidth]{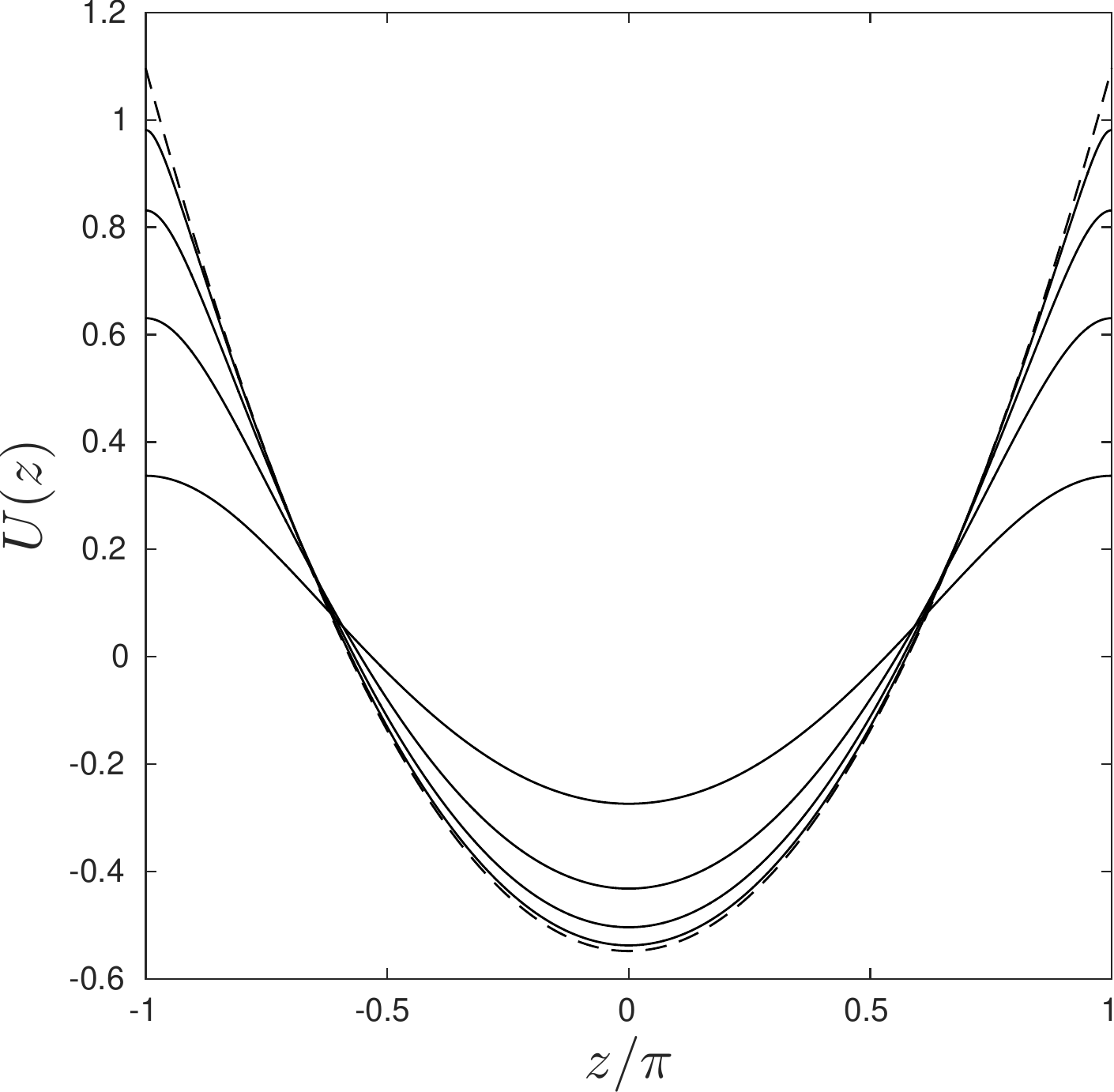}
    \hfill
    \includegraphics[width=0.46\textwidth]{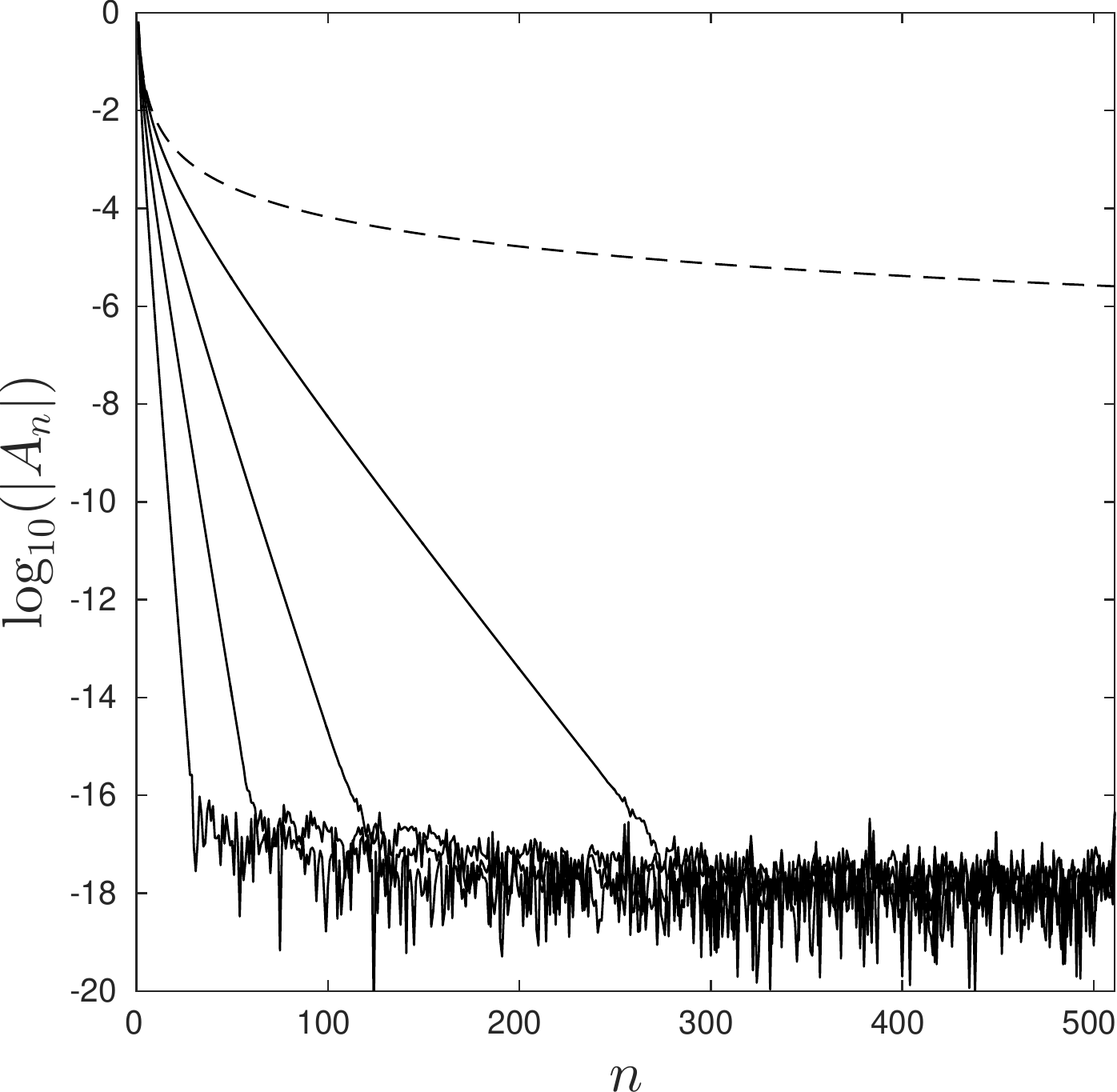}\\
    \null\hfill(a)\hspace{0.5\textwidth}(b)\hfill\null
    \caption{(a) The $2\pi$-periodic solutions of the reduced Ostrovsky equation (\ref{second-order}) for $a = -0.3, -0.5, -0.6, -0.65$.
    (b) The logarithm of the absolute value of the Fourier cosine coefficients, $A_n$, of the trigonometric approximation (\ref{partial-sum}).
    The limiting piecewise parabolic wave, corresponding to $a=-\tfrac{2}{3}$ is shown dashed in (a) with the corresponding coefficients $A_n=2(-1)^n/3n^2$ included in (b). \label{f:RO} }
  \end{figure}

To discuss the eigenvalues of the operator $\PP(\kappa)$ given by \eqref{unperturbed-operator-perturbed}, it is convenient to write
\begin{equation}
  \PP(\kappa) = \AP(\kappa) - c \BP(\kappa),
  \label{Pnum}
\end{equation}
where
\begin{eqnarray*}
   \AP(\kappa) &=&  -(\partial_z + i \kappa)^{-2} - (\gamma - U), \\
   \BP(\kappa) &=&  (\gamma^3 - 6I)^{-2/3} - (\gamma^3 - 6I)^{-5/3} (\partial_z+i\kappa)^2 (\gamma - U)^5 (\partial_z+i \kappa)^2.
\end{eqnarray*}
By discretising the linear operators in Fourier space and evaluating products pseudospectrally, we obtain the discretised forms
\begin{eqnarray*}
   \widehat{\AP}(\kappa) &=&  {\rm diag}(\bmkk^2)-\FF({\rm diag}(\gamma-\bmU)\FF^{-1}(\II)), \\
   \widehat{\BP}(\kappa) &=&  (\gamma^3 - 6I)^{-2/3}\II - (\gamma^3 - 6I)^{-5/3}{\rm diag}(\bmk^2)\FF({\rm diag}(\gamma - \bmU)^5 \FF^{-1}({\rm diag}(\bmk^2)),
\end{eqnarray*}
where $\FF$ and $\FF^{-1}$ denote the discrete Fourier transform and its inverse, $\bmk$ is the wavenumber vector with components $\kappa\pm n$ and $\bmkk$ its component-wise inverse. Eigenvalues were obtained from the discretized form of the operators using the Matlab subroutines {\tt eig} and {\tt eigs}.

\begin{figure}[!htbp]
    \includegraphics[width=0.48\textwidth]{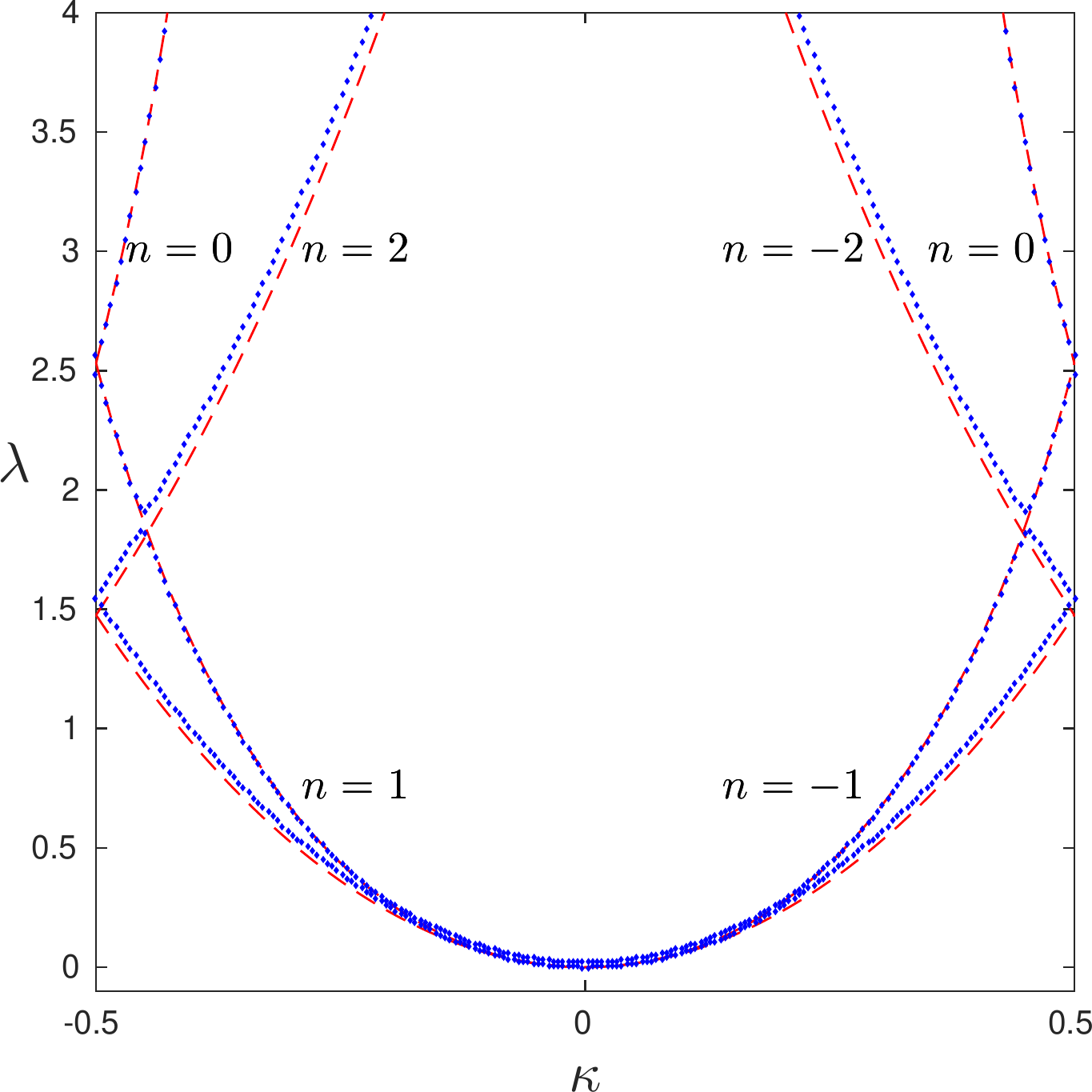} \hfill
    \includegraphics[width=0.48\textwidth]{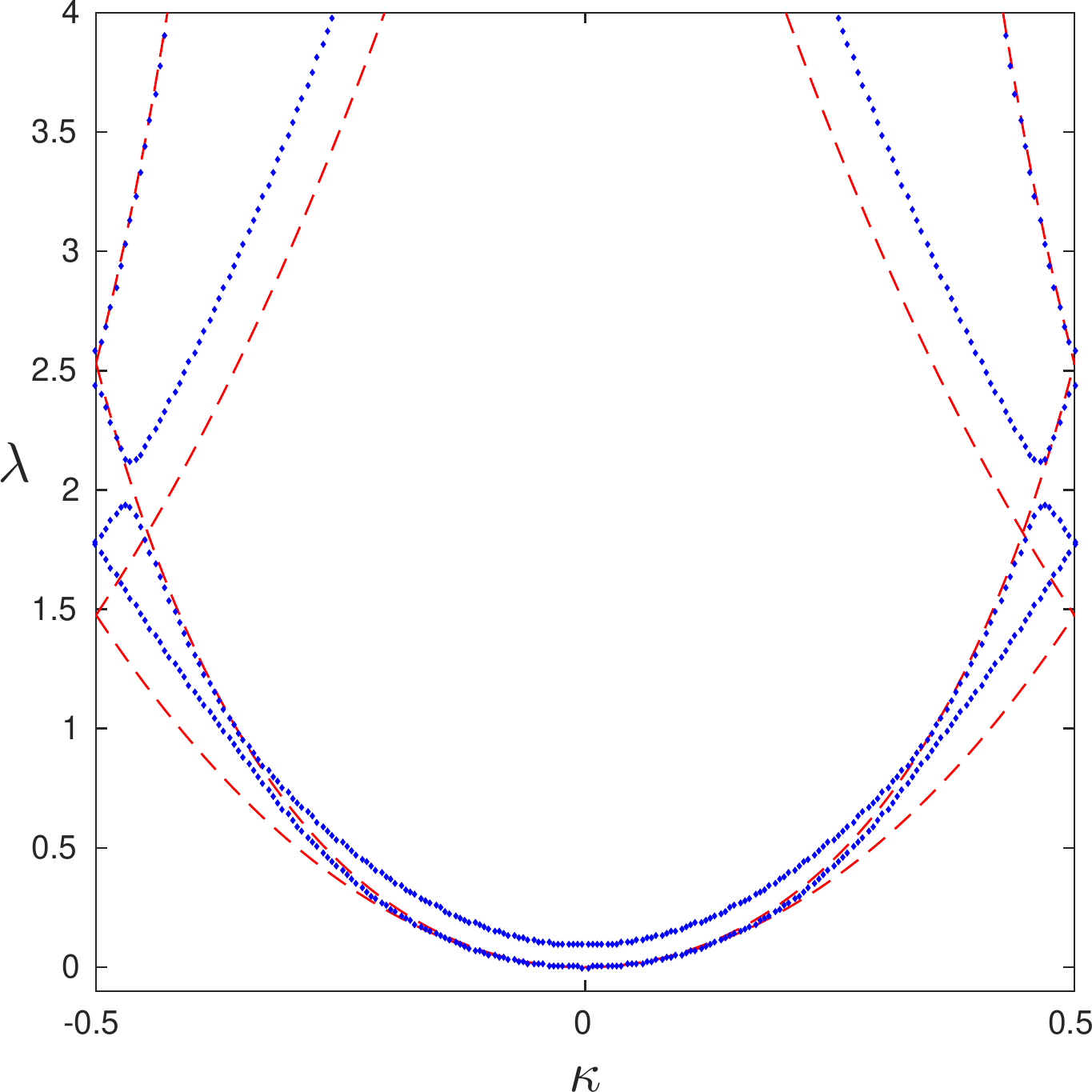} \\
        \null\hfill(a)\hspace{0.5\textwidth}(b)\hfill\null
    \caption{The lowest eigenvalues of the operator \eqref{Pnum} as a function of $\kappa$ when $c=0.5$
    for (a) $a=-0.1$ and (b) $a=-0.2$. The dashed lines (red online) give the lowest eigenvalues 
    of the unperturbed operator for $a = 0$ and the computed eigenvalues are shown as diamonds (blue online).
    All repeated eigenvalues for $a = 0$ are split as $a \neq 0$. \label{f:a0bands} }
\end{figure}

\begin{figure}[!htbp]
    \includegraphics[width=0.31\textwidth]{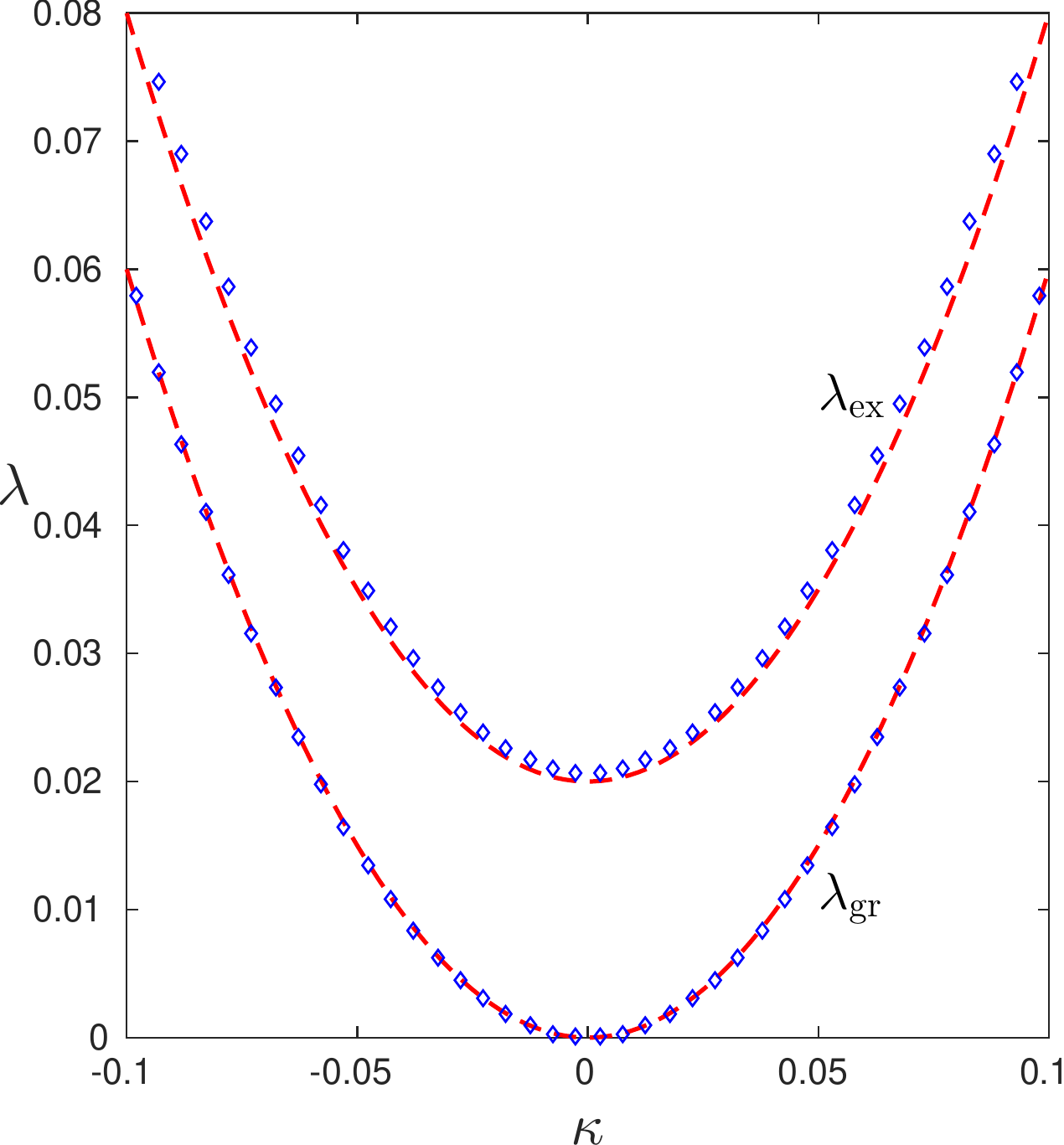} \hfill
   \includegraphics[width=0.32\textwidth]{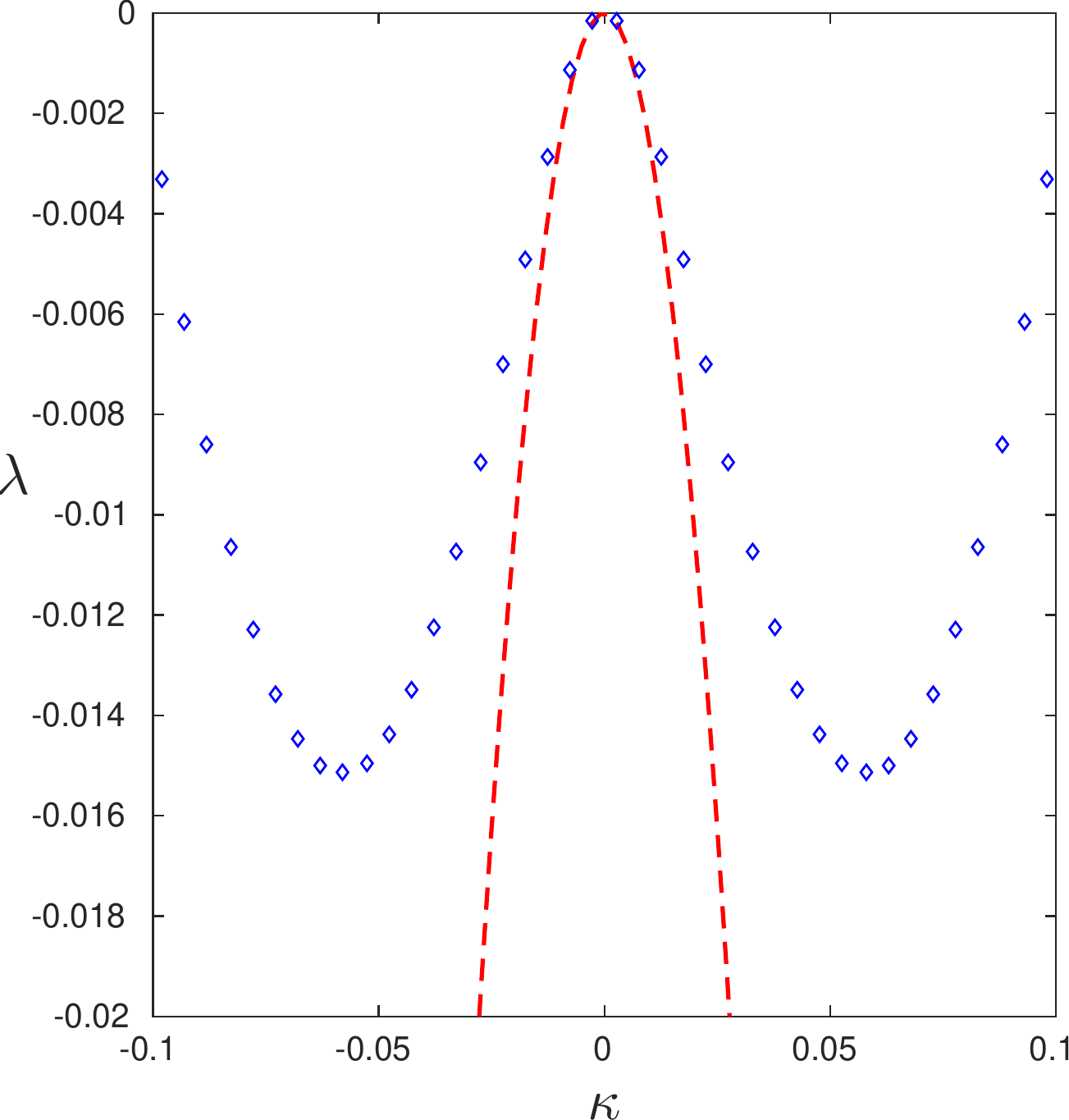} \hfill
    \includegraphics[width=0.31\textwidth]{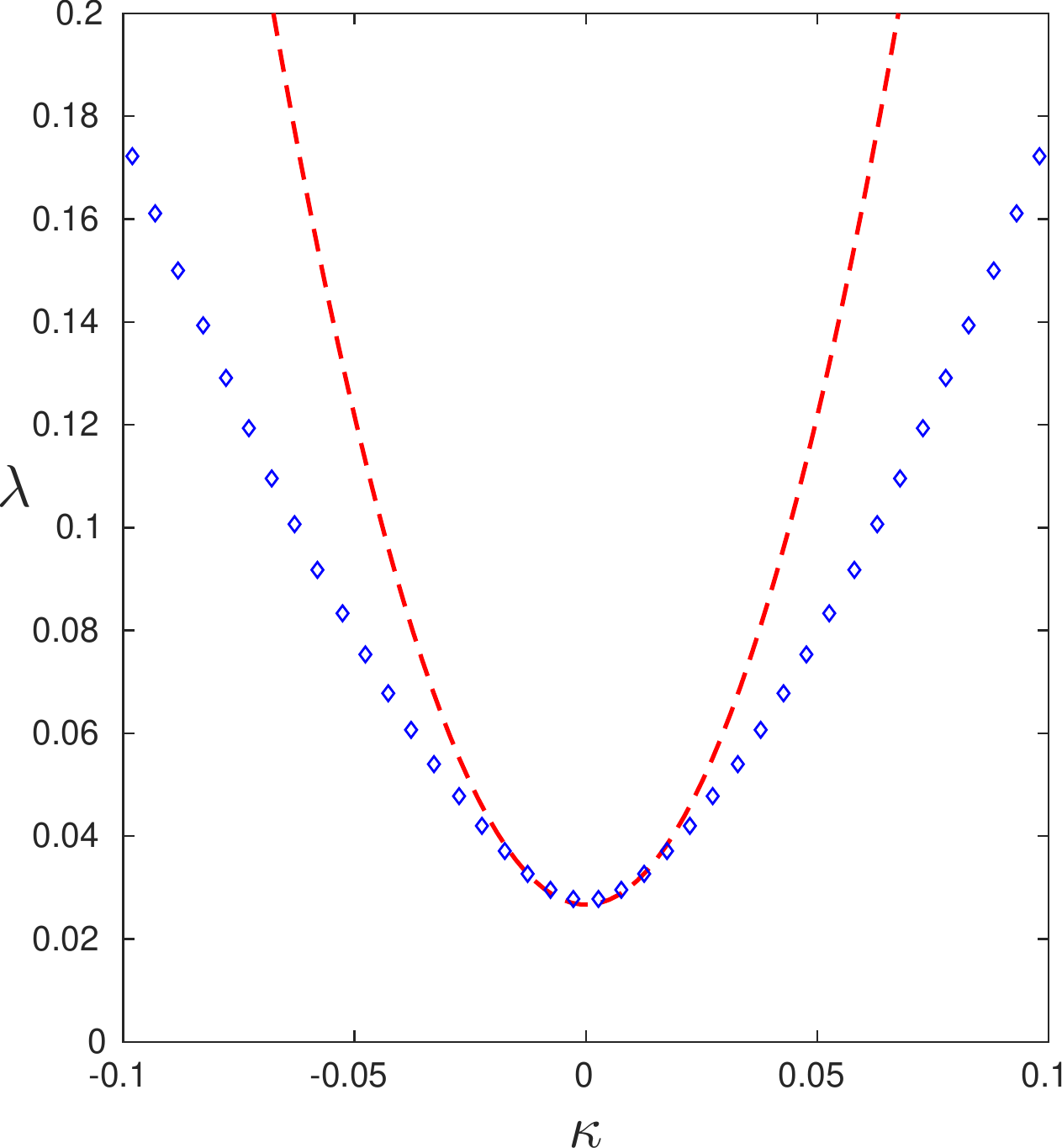} \\
        \null\hfill(a)\hspace{0.31\textwidth}(b)\hspace{0.31\textwidth}(c)\hfill\null
    \caption{(a) A detail of figure \ref{f:a0bands}(a) (for $c=0.5$ and $a=-0.1$) in the neighbourhood
    of the origin showing the splitting of the two spectral bands at finite $a$.
    (b) The ground spectral band for $a=-0.1$ but for $c=0.7$. (c) The first excited spectral band
    for $a=-0.1$ and $c=0.7$. The dashed lines here give the small-$\kappa$, small-$a$ asymptotic
     expansions \eqref{band-perturbation-lowest} and \eqref{band-perturbation-second} for the two 
     spectral bands. \label{f:ap1bands}}
\end{figure}

Figure \ref{f:a0bands} compares the computed lowest eigenvalues of the operator (\ref{Pnum}) for amplitudes $a=-0.1$ and $a=-0.2$
with the spectral bands for the unperturbed operator for $a = 0$ at $c = c_0 = 0.5$. At finite amplitudes ($a\neq0$), all
repeated eigenvalues of the unperturbed operator are split, including the repeated eigenvalue at the origin. Figure \ref{f:ap1bands}(a) is
a detail of figure \ref{f:a0bands}(a) in the neighbourhood of the origin, with the dashed lines now giving the
small-$\kappa$, small-$a$ asymptotic expansions \eqref{band-perturbation-lowest} and \eqref{band-perturbation-second}
for the ground and first excited spectral bands with $\lambda''(0) = 12$, $\mu_1 = 4$, and $\lambda_2(c_0) = 2$.
As predicted, the excited spectral band moves symmetrically upwards into positive $\lambda$, with the asymptotic solution
remaining accurate even at significant base-wave amplitudes. Figure \ref{f:ap1bands}(b),(c) gives the ground and excited
spectral bands for $a = -0.1$ and $c=0.7$. In line with the expansion \eqref{band-perturbation-lowest-red-Ost},
the value of $c$ is now sufficiently large that the lowest mode curve is concave downwards in the neighbourhood of the origin,
mirrored by the appearance of computed eigenvalues with $\lambda$ negative. The accuracy of the asymptotic forms for $a=-0.1$ is remarkable.

Computations for all allowable $0 < a < 0.65$ and $0<c<0.7$ show that the small $\kappa$ behaviour of
the expansion \eqref{band-perturbation-lowest-red-Ost} shown on figure \ref{f:ap1bands} is generic.
At fixed $a$ the graph of the spectral band $\lambda_{\rm gr}(\kappa)$ is concave upwards as
a function of $\kappa$ for $c\in(c_-,c_+)$ and concave downwards outside this interval. Moreover
this change is the first occurrence of a negative eigenvalue for $P_{c,\gamma}(\kappa)$.
Thus the boundaries $c_\pm$ are determined by changes in sign of $\lambda_{\rm gr}''(0)$. Since
$\lambda_{\rm gr}'(0)=0$,  it is convenient to determine the sign of $\lambda_{\rm gr}''(0)$
by the sign of $\lambda_{\rm gr}''(\delta_\kappa)$ for $0<\delta_\kappa\ll1$.
The boundaries $c_\pm$ are thus determined as the values of $c$ for which $\PP(\delta_\kappa)$
is not invertible, i.e eigenvalues of the generalised linear eigenvalue problem
\begin{equation}
  \AP(\delta_\kappa) = c \BP(\delta_\kappa).
\end{equation}
The computations reported here were performed for $\delta_\kappa=10^{-2},10^{-3},10^{-4}$ and
the results were graphically indistinguishable. Figure \ref{f:P_pos}(a) shows shaded the region of
the $(c,|a|)$ plane where the operator $\PP(\kappa)$ is positive for all $\kappa$, and the
accuracy of the  asymptotic form for the boundary for small $a$ given by the expansion
\eqref{exact-c-plus-minus} in Theorem \ref{theorem-red-Ost}. Computations for the shaded region
have been performed for all $|a|\leq0.66$ to show that the region of positivity of $\PP(\kappa)$
extends effectively to the maximum amplitude of $|a|=2/3$.

Numerical computations for the modified reduced Ostrovsky equation (\ref{redModOst})
proceed analogously to those for the reduced Ostrovsky equation (\ref{redOst}).
Thus, the results will be omitted here except for figure \ref{f:P_pos}(b)
which gives the region of the $(c,|a|)$ plane in which the operator $\PP(\kappa)$
defined by (\ref{unperturbed-operator-perturbed-modified}) is positive for all $\kappa$.
Again, the asymptotic expansion for the boundary for small $a$ given by \eqref{exact-c-plus-minus-mod}
in Theorem \ref{theorem-red-Mod-Ost} is very accurate. Computations for the shaded region
have been performed for all $|a|\leq 1.27$ to show that the region of positivity of $\PP(\kappa)$
extends effectively to the maximum amplitude of $|a|=4/\pi\approx1.273$.

\begin{figure}
    \includegraphics[width=0.49\textwidth]{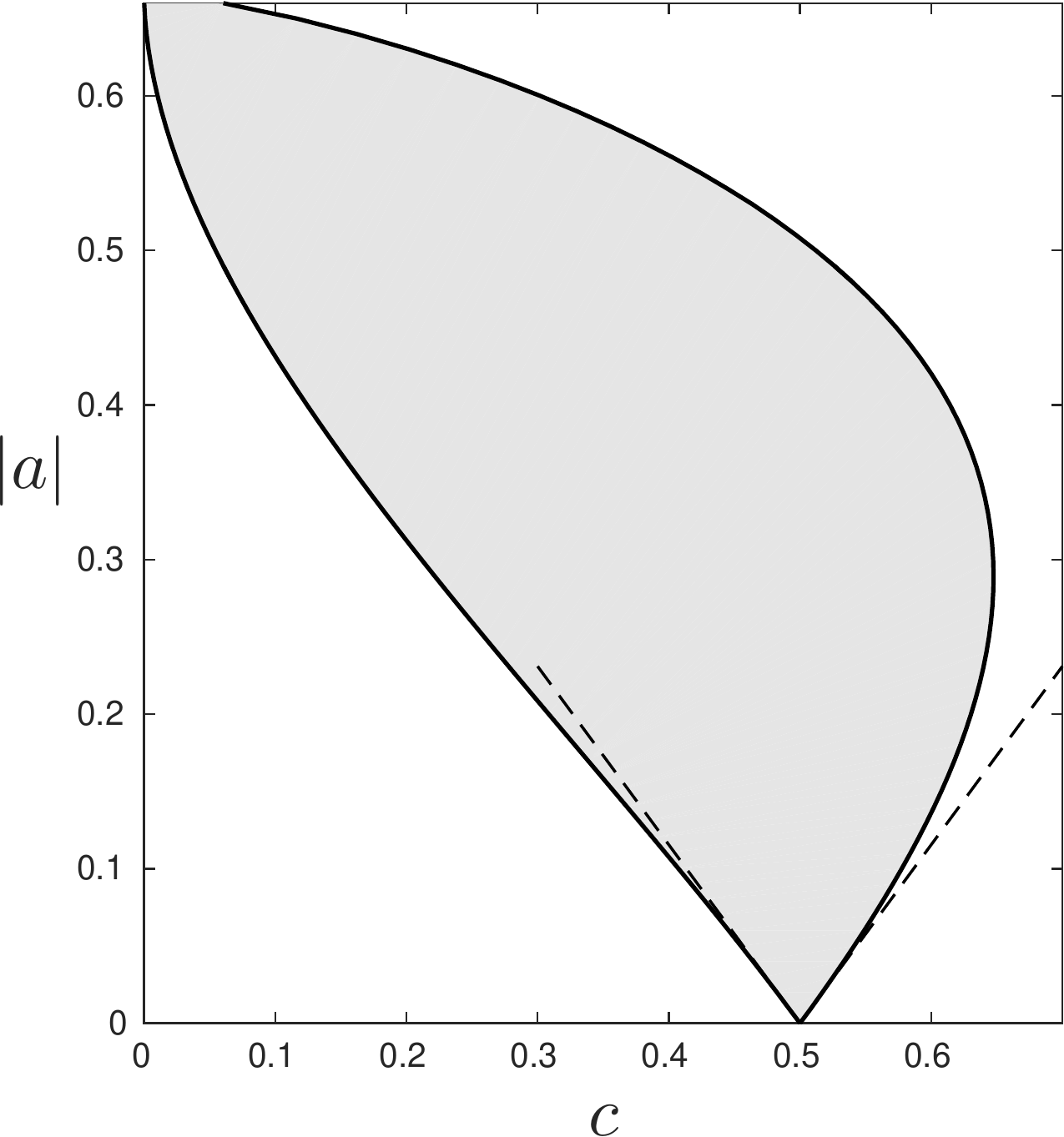}\hfill
    \includegraphics[width=0.49\textwidth]{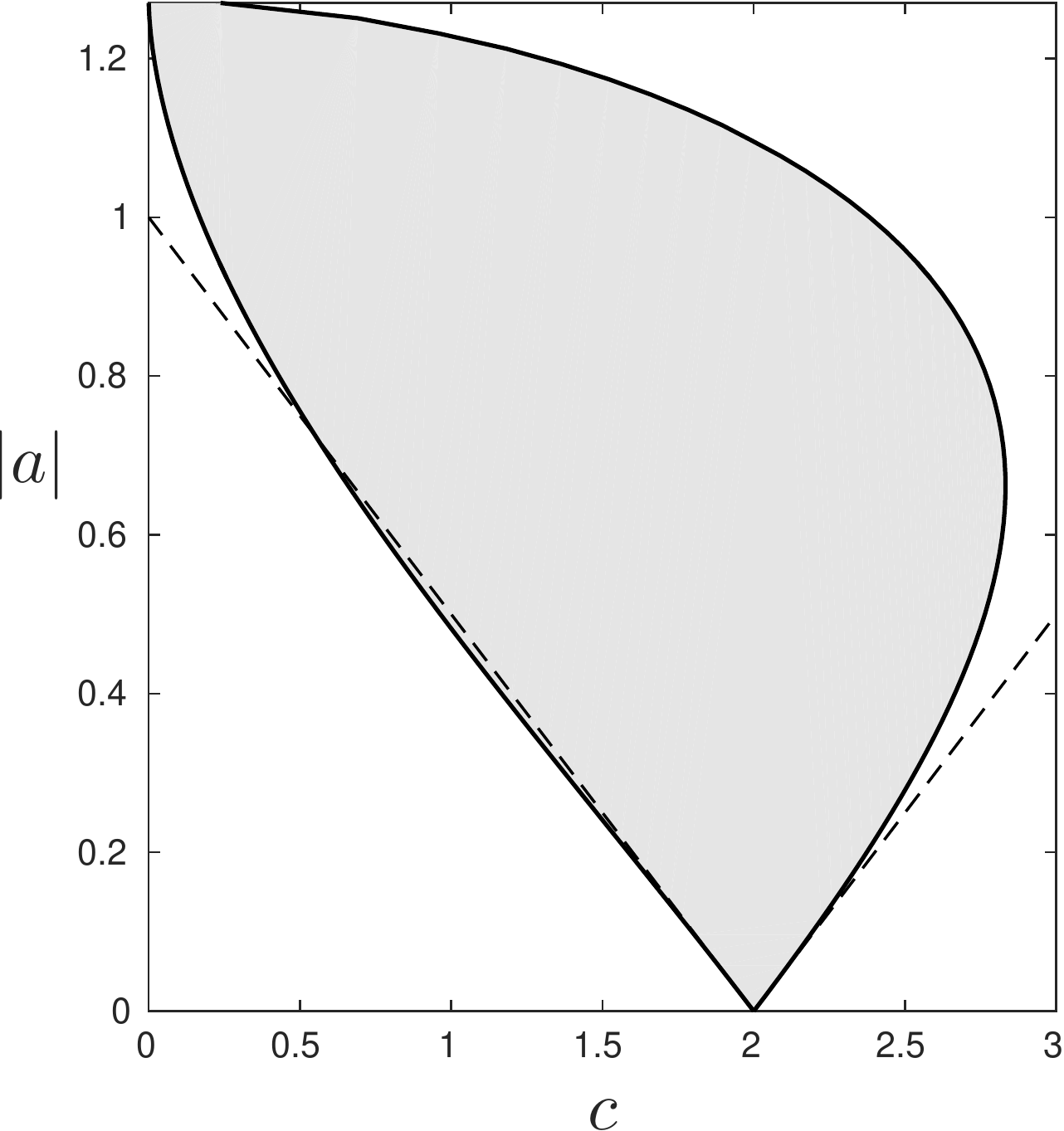}\\
    \null\hfill(a)\hspace{0.5\textwidth}(b)\hfill\null
    \caption{The  region of the $(c,|a|)$ plane where the operator $\PP(\kappa)$ is positive
    for all $\kappa$ is shown shaded for the reduced Ostrovsky equation (a) and
     the reduced modified Ostrovsky equation (b). The dashed lines show the asymptotic expansions
     for the boundary for small $|a|$ as given by \eqref{exact-c-plus-minus} and \eqref{exact-c-plus-minus-mod}
     respectively. \label{f:P_pos} }
\end{figure}

\section{Discussion}

The short-pulse equation (\ref{shortPulse}) can be considered as the focusing version
of the modified reduced Ostrovsky equation (\ref{redModOst}).
The short-pulse equation (\ref{shortPulse}) possesses the modulated pulse solutions,
which are localized in space and periodic in time \cite{SS06}. These solutions
can have arbitrary small amplitude and wide localization in space, when the solution
resembles modulated wave packets governed by the focusing nonlinear Schr\"{o}dinger
equation \cite{GrimPP}. Therefore, it is natural to suspect that the small-amplitude periodic waves
are unstable with respect to side-band modulations \cite{Ostrovsky}.
Although the short-pulse equation also possess higher-order conserved quantities \cite{Br,PelSak},
our method relying on construction of a Lyapunov-type energy functional 
should fail for periodic waves in the short-pulse equation
(\ref{shortPulse}). Here we show how precisely the method fails.

The periodic wave given by (\ref{trav-wave-mod}) satisfies the second-order differential equation
\begin{equation}
\label{second-order-short-pulse}
\frac{d}{dz} \left[ \left( \gamma + \frac{1}{2} U^2 \right) \frac{dU}{dz} \right] + U(z) = 0,
\end{equation}
which has no constraints on the amplitude of periodic waves. Looking at the Stokes expansions (\ref{Stokes-mod}),
we obtain the existence of $2\pi$-periodic smooth solutions $U$ of the differential equation (\ref{second-order-short-pulse})
for parameter $\gamma < 1$ satisfying the asymptotic expansion
\begin{equation}
\label{gamma-parameter-short-pulse}
\gamma = 1 - \frac{1}{8} a^2 + \mathcal{O}(a^4).
\end{equation}

Periodic waves are critical points of the two energy functionals
\begin{equation}
S_{\gamma}(u) := \| \partial_x^{-1} u \|_{L^2}^2 - \frac{1}{12} \| u \|_{L^4}^4 - \gamma \| u \|_{L^2}^2
\end{equation}
and
\begin{equation}
R_{\Gamma}(u) := \int (1 + u_x^2)^{1/2} dz - \frac{1}{2 (\gamma^2 + 2 I)^{1/2}} \| u \|_{L^2}^2,
\end{equation}
where $I$ is the first-order invariant associated with the differential equation (\ref{second-order-short-pulse}).

The second variations of the energy functionals $S_{\gamma}$ and $R_{\Gamma}$ are defined by
the linear operators $L_{\gamma}$ and $M_{\gamma}$ in the following explicit form:
$$
L_{\gamma} := -\partial_z^{-2} - \gamma - \frac{1}{2} U^2 : \; L^2_{\rm per,zero} \to L^2_{\rm per,zero}.
$$
and
$$
M_{\gamma} := -\frac{1}{2} (\gamma^2 + 2 I)^{-1/2} - \frac{1}{2} \partial_z (1+(U')^2)^{-3/2} \partial_z : \;
H^2_{\rm per, zero} \to L^2_{\rm per, zero}.
$$
Compared to the operators $L_{\gamma = 1}$ and $M_{\gamma = 1}$ in Lemmas \ref{lemma-spectrum-S-mod}
and \ref{lemma-spectrum-R-mod}, the operator $M_{\gamma = 1}$ has an infinite number of positive eigenvalues 
and a finite number of negative eigenvalues.
The splitting of the zero eigenvalue is studied by the regular asymptotic expansion (\ref{expansion-same}).
For the operator $L_{\gamma}$, we obtain $\lambda_2 = -\frac{1}{4}$ instead of (\ref{lambda-2-S-mod}).
For the operator $M_{\gamma}$, we still obtain $\lambda_2 = -\frac{3}{8}$ as in (\ref{lambda-2-R-mod}).

Defining $\Lambda_{c,\gamma}(u) := S_{\gamma}(u) + c R_{\Gamma}(u)$ to reflect the change in the sign for $M_{\gamma}$,
we obtain the result of Lemma \ref{lemma-positivity-mod} for $c = c_0 = 2$. Therefore, Assumption \ref{assumption-1} is still satisfied.
However, expansion (\ref{eigenvalue-positive}) of Assumption \ref{assumption-2} gives now the negative eigenvalue $\lambda(a,c)$
with
$$
\lambda_2(c) = -\frac{2+3c}{8}.
$$
Therefore, one of the spectral bands of the linear operator $K_{c,\gamma} := L_{\gamma} + c M_{\gamma}$
is now negative near $\kappa = 0$ for every small nonzero amplitude $a$ of the periodic wave with the profile $U$.
As a result, $\Lambda_{c,\gamma}$ is not positive for every $c \in \mathbb{R}$ if $|\gamma - 1|$ is
sufficiently small. This indicates that $\Lambda_{c,\gamma}$ is no longer a Lyapunov-type energy functional
for the periodic waves of the short-pulse equation (\ref{shortPulse}), which are modulationally unstable.

As an open problem, we mention that the modulated pulse solutions of the short-pulse equation (\ref{shortPulse})
are reported to be stable in numerical simulations \cite{Schafer1,LPS2}. Given integrability structure
of the short-pulse equation, it may be possible to prove orbital stability of the modulated pulse solutions analytically.
A similar proof of nonlinear orbital stability of breathers in the modified KdV equation was recently developed in \cite{Munoz}.
However, there are technical obstacles to extend this proof to breathers in the sine--Gordon equation \cite{AM},
and hence to the short-pulse equation (\ref{shortPulse}), which corresponds to
the sine--Gordon equation in characteristic coordinates \cite{PelSak}.

\vspace{0.5cm}

\noindent{\bf Acknowledgements.} D.P. is supported by the LMS Visiting Scheme 2.
He thanks members of the Institute of Mathematics at UCL for
hospitality during his visit (May-June, 2015).

\end{document}